\documentclass[12pt]{amsart}       
\usepackage{txfonts}
\usepackage{amssymb}
\usepackage{eucal}
\usepackage{graphicx}
\usepackage{amsmath}
\usepackage{amscd}
\usepackage[all]{xy}           

\usepackage{amsfonts,latexsym}
\usepackage{xspace}
\usepackage{epsfig}
\usepackage{float}
\usepackage{color}
\usepackage{fancybox}
\usepackage{colordvi}
\usepackage{multicol}
\usepackage{colordvi}
\usepackage{tikz}
\usepackage[colorlinks,final,backref=page,hyperindex,hypertex]{hyperref}
\usepackage[active]{srcltx} 

\newtheorem{theorem}{Theorem}[section]
\newtheorem{prop}[theorem]{Proposition}
\newtheorem{lemma}[theorem]{Lemma}

\newtheorem{prop-def}[theorem]{Proposition-Definition}
\newtheorem{coro-def}[theorem]{Corollary-Definition}

\newtheorem{algo}[theorem]{Algorithm}
\newtheorem{question}[theorem]{Question}
\theoremstyle{definition}
\newtheorem{defn}[theorem]{Definition}
\newtheorem{remark}[theorem]{Remark}

\newtheorem{exam}[theorem]{Example}

\newcommand{\nc}{\newcommand}
\nc{\tred}[1]{\textcolor{red}{#1}}
\nc{\tblue}[1]{\textcolor{blue}{#1}}
\nc{\tgreen}[1]{\textcolor{green}{#1}}
\nc{\tpurple}[1]{\textcolor{purple}{#1}}
\nc{\btred}[1]{\textcolor{red}{\bf #1}}
\nc{\btblue}[1]{\textcolor{blue}{\bf #1}}
\nc{\btgreen}[1]{\textcolor{green}{\bf #1}}
\nc{\btpurple}[1]{\textcolor{purple}{\bf #1}}

\renewcommand{\frak}{\mathfrak}

\newcommand{\efootnote}[1]{}

\renewcommand{\textbf}[1]{}

\newcommand{\delete}[1]{}

\nc{\mlabel}[1]{\label{#1}}  
\nc{\mcite}[1]{\cite{#1}}  
\nc{\mref}[1]{\ref{#1}}  
\nc{\mbibitem}[1]{\bibitem{#1}} 

\delete{
\nc{\mlabel}[1]{\label{#1}  
{\hfill \hspace{1cm}{\small\tt{{\ }\hfill(#1)}}}}
\nc{\mcite}[1]{\cite{#1}{\small{\tt{{\ }(#1)}}}}  
\nc{\mref}[1]{\ref{#1}{{\tt{{\ }(#1)}}}}  
\nc{\mbibitem}[1]{\bibitem[\bf #1]{#1}} 
}

\nc{\opa}{\ast} \nc{\opb}{\odot} \nc{\op}{\bullet} \nc{\pa}{\frakL}
\nc{\arr}{\rightarrow} \nc{\lu}[1]{(#1)} \nc{\mult}{\mrm{mult}}
\nc{\diff}{\mathfrak{Diff}}
\nc{\opc}{\sharp}\nc{\opd}{\natural}
\nc{\ope}{\circ}
\nc{\dpt}{\mathrm{d}}

\nc{\tforall}{\ \text{   for all }}
\nc{\AW}{\mathcal{A}}

\nc{\ari}{\mathrm{ar}}

\nc{\lef}{\mathrm{lef}}

\nc{\Sh}{\mathrm{ST}}

\nc{\Cr}{\mathrm{Cr}}

\nc{\st}{{Schr\"oder tree}\xspace}
\nc{\sts}{{Schr\"oder trees}\xspace}

\nc{\vertset}{\Omega} 
\nc{\dep}{\mrm{dep}}

\nc{\assop}{\quad \begin{picture}(5,5)(0,0)
\line(-1,1){10}
\put(-2.2,-2.2){$\bullet$}
\line(0,-1){10}\line(1,1){10}
\end{picture} \quad \smallskip}

\nc{\operator}{\begin{picture}(5,5)(0,0)
\line(0,-1){6}
\put(-2.6,-1.8){$\bullet$}
\line(0,1){9}
\end{picture}}

\nc{\idx}{\begin{picture}(6,6)(-3,-3)
\put(0,0){\line(0,1){6}}
\put(0,0){\line(0,-1){6}}
 \end{picture}}

\nc{\pb}{{\mathrm{pb}}}
\nc{\Lf}{{\mathrm{Lf}}}

\nc{\lft}{{left tree}\xspace}
\nc{\lfts}{{left trees}\xspace}

\nc{\fat}{{fundamental averaging tree}\xspace}

\nc{\fats}{{fundamental averaging trees}\xspace}
\nc{\avt}{\mathrm{Avt}}

\nc{\rass}{{\mathit{RAss}}}

\nc{\aass}{{\mathit{AAss}}}

\nc{\vin}{{\mathrm Vin}}    
\nc{\lin}{{\mathrm Lin}}    
\nc{\inv}{\mathrm{I}n}
\nc{\gensp}{V} 
\nc{\genbas}{\mathcal{V}} 
\nc{\bvp}{V_P}     
\nc{\gop}{{\,\omega\,}}     

\nc{\bin}[2]{ (_{\stackrel{\scs{#1}}{\scs{#2}}})}  
\nc{\binc}[2]{ \left (\!\! \begin{array}{c} \scs{#1}\\
    \scs{#2} \end{array}\!\! \right )}  
\nc{\bincc}[2]{  \left ( {\scs{#1} \atop
    \vspace{-1cm}\scs{#2}} \right )}  
\nc{\bs}{\bar{S}} \nc{\cosum}{\sqsubset} \nc{\la}{\longrightarrow}
\nc{\rar}{\rightarrow} \nc{\dar}{\downarrow} \nc{\dprod}{**}
\nc{\dap}[1]{\downarrow \rlap{$\scriptstyle{#1}$}}
\nc{\md}{\mathrm{dth}} \nc{\uap}[1]{\uparrow
\rlap{$\scriptstyle{#1}$}} \nc{\defeq}{\stackrel{\rm def}{=}}
\nc{\disp}[1]{\displaystyle{#1}} \nc{\dotcup}{\
\displaystyle{\bigcup^\bullet}\ } \nc{\gzeta}{\bar{\zeta}}
\nc{\hcm}{\ \hat{,}\ } \nc{\hts}{\hat{\otimes}}
\nc{\barot}{{\otimes}} \nc{\free}[1]{\bar{#1}}
\nc{\uni}[1]{\tilde{#1}} \nc{\hcirc}{\hat{\circ}} \nc{\lleft}{[}
\nc{\lright}{]} \nc{\lc}{\lfloor} \nc{\rc}{\rfloor}
\nc{\curlyl}{\left \{ \begin{array}{c} {} \\ {} \end{array}
    \right .  \!\!\!\!\!\!\!}
\nc{\curlyr}{ \!\!\!\!\!\!\!
    \left . \begin{array}{c} {} \\ {} \end{array}
    \right \} }
\nc{\longmid}{\left | \begin{array}{c} {} \\ {} \end{array}
    \right . \!\!\!\!\!\!\!}
\nc{\onetree}{\bullet} \nc{\ora}[1]{\stackrel{#1}{\rar}}
\nc{\ola}[1]{\stackrel{#1}{\la}}
\nc{\ot}{\otimes} \nc{\mot}{{{\boxtimes\,}}}
\nc{\otm}{\overline{\boxtimes}} \nc{\sprod}{\bullet}
\nc{\scs}[1]{\scriptstyle{#1}} \nc{\mrm}[1]{{\rm #1}}
\nc{\margin}[1]{\marginpar{\rm #1}}   
\nc{\dirlim}{\displaystyle{\lim_{\longrightarrow}}\,}
\nc{\invlim}{\displaystyle{\lim_{\longleftarrow}}\,}
\nc{\mvp}{\vspace{0.3cm}} \nc{\tk}{^{(k)}} \nc{\tp}{^\prime}
\nc{\ttp}{^{\prime\prime}} \nc{\svp}{\vspace{2cm}}
\nc{\vp}{\vspace{8cm}} \nc{\proofbegin}{\noindent{\bf Proof: }}
\nc{\proofend}{$\blacksquare$ \vspace{0.3cm}}
\nc{\modg}[1]{\!<\!\!{#1}\!\!>}
\nc{\intg}[1]{F_C(#1)} \nc{\lmodg}{\!
<\!\!} \nc{\rmodg}{\!\!>\!}
\nc{\cpi}{\widehat{\Pi}}
\nc{\sha}{{\mbox{\cyr X}}}  
\nc{\shap}{{\mbox{\cyrs X}}} 
\nc{\shpr}{\diamond}    
\nc{\shp}{\ast} \nc{\shplus}{\shpr^+}
\nc{\shprc}{\shpr_c}    
\nc{\msh}{\ast} \nc{\zprod}{m_0} \nc{\oprod}{m_1}
\nc{\vep}{\varepsilon} \nc{\labs}{\mid\!} \nc{\rabs}{\!\mid}

\nc{\mmbox}[1]{\mbox{\ #1\ }} \nc{\fp}{\mrm{FP}}
\nc{\rchar}{\mrm{char}} \nc{\End}{\mrm{End}} \nc{\Fil}{\mrm{Fil}}
\nc{\Mor}{Mor\xspace} \nc{\gmzvs}{gMZV\xspace}
\nc{\gmzv}{gMZV\xspace} \nc{\mzv}{MZV\xspace}
\nc{\mzvs}{MZVs\xspace} \nc{\Hom}{\mrm{Hom}} \nc{\id}{\mrm{id}}
\nc{\im}{\mrm{im}} \nc{\incl}{\mrm{incl}} \nc{\map}{\mrm{Map}}
\nc{\mchar}{\rm char} \nc{\nz}{\rm NZ} \nc{\supp}{\mathrm Supp}
 \nc{\bre}{\mrm{b}}
\nc{\Alg}{\mathbf{Alg}} \nc{\Bax}{\mathbf{Bax}} \nc{\bff}{\mathbf f}
\nc{\bfk}{{\mathbf k}} \nc{\bfone}{{\bf 1}} \nc{\bfx}{\mathbf x}
\nc{\bfy}{\mathbf y}
\nc{\base}[1]{\bfone^{\otimes ({#1}+1)}} 
\nc{\Cat}{\mathbf{Cat}}

\nc{\detail}{\marginpar{\bf More detail}
    \noindent{\bf Need more detail!}
    \svp}
\nc{\Int}{\mathbf{Int}} \nc{\Mon}{\mathbf{Mon}}
\nc{\rbtm}{{shuffle }} \nc{\rbto}{{Rota-Baxter }}
\nc{\remarks}{\noindent{\bf Remarks: }} \nc{\Rings}{\mathbf{Rings}}
\nc{\Sets}{\mathbf{Sets}} \nc{\wtot}{\widetilde{\odot}}
\nc{\wast}{\widetilde{\ast}} \nc{\bodot}{\bar{\odot}}
\nc{\bast}{\bar{\ast}} \nc{\hodot}[1]{\odot^{#1}}
\nc{\hast}[1]{\ast^{#1}} \nc{\mal}{\mathcal{O}}
\nc{\tet}{\tilde{\ast}} \nc{\teot}{\tilde{\odot}}
\nc{\oex}{\overline{x}} \nc{\oey}{\overline{y}}
\nc{\oez}{\overline{z}} \nc{\oef}{\overline{f}}
\nc{\oea}{\overline{a}} \nc{\oeb}{\overline{b}}
\nc{\weast}[1]{\widetilde{\ast}^{#1}}
\nc{\weodot}[1]{\widetilde{\odot}^{#1}} \nc{\hstar}[1]{\star^{#1}}
\nc{\lae}{\langle} \nc{\rae}{\rangle}
\nc{\lf}{\lfloor}\nc{\rf}{\rfloor}

\nc{\QQ}{{\mathbb Q}}\nc{\NN}{{\mathbb N}}
\nc{\RR}{{\mathbb R}} \nc{\ZZ}{{\mathbb Z}}

\nc{\cala}{{\mathcal A}} \nc{\calb}{{\mathcal B}}
\nc{\calc}{{\mathcal C}}
\nc{\cald}{{\mathcal D}} \nc{\cale}{{\mathcal E}}
\nc{\calf}{{\mathcal F}} \nc{\calg}{{\mathcal G}}
\nc{\calh}{{\mathcal H}} \nc{\cali}{{\mathcal I}}
\nc{\call}{{\mathcal L}} \nc{\calm}{{\mathcal M}}
\nc{\caln}{{\mathcal N}} \nc{\calo}{{\mathcal O}}
\nc{\calp}{{\mathcal P}} \nc{\calr}{{\mathcal R}}
\nc{\cals}{{\mathcal S}} \nc{\calt}{{\mathcal T}}
\nc{\calu}{{\mathcal U}} \nc{\calw}{{\mathcal W}} \nc{\calk}{{\mathcal K}}
\nc{\calx}{{\mathcal X}} \nc{\CA}{\mathcal{A}}

\nc{\fraka}{{\mathfrak a}} \nc{\frakA}{{\mathfrak A}}
\nc{\frakb}{{\mathfrak b}} \nc{\frakB}{{\mathfrak B}}
\nc{\frakD}{{\mathfrak D}} \nc{\frakg}{{\mathfrak g}}
\nc{\frakH}{{\mathfrak H}} \nc{\frakL}{{\mathfrak L}}
\nc{\frakM}{{\mathfrak M}} \nc{\bfrakM}{\overline{\frakM}}
\nc{\frakm}{{\mathfrak m}} \nc{\frakP}{{\mathfrak P}}
\nc{\frakN}{{\mathfrak N}} \nc{\frakp}{{\mathfrak p}}
\nc{\frakR}{{\mathfrak R}} \nc{\frakS}{{\mathfrak S}}

\nc{\BS}{\mathbb{S}}

\font\cyr=wncyr10 \font\cyrs=wncyr7
\nc{\li}[1]{\textcolor{red}{#1}}
\nc{\lir}[1]{\textcolor{red}{Li:#1}}
\nc{\tj}[1]{\textcolor{blue}{Tianjie: #1}}
\nc{\red}[1]{\textcolor{red}{#1}}
\nc{\blue}[1]{\textcolor{blue}{#1}}
\nc{\green}[1]{\textcolor{green}{#1}}
\nc{\yellow}[1]{\textcolor{yellow}{#1}}
\nc{\purple}[1]{\textcolor{purple}{#1}}

\begin{document}

\title[Reynolds algebras and their free objects]{Reynolds algebras and their free objects from bracketed words and rooted trees}
%
\author{Tianjie Zhang}
\address{School of Mathematics and Statistics, Ningxia University, Yinchuan, Ningxia 750021, China}
\email{tjzhangmath@aliyun.com}

\author{Xing Gao}
\address{Department of Mathematics, Lanzhou University, Lanzhou, Gansu 730000, China}
\email{gaoxing@lzu.edu.cn}

\author{Li Guo}
\address{Department of Mathematics and Computer Science, Rutgers University, Newark, NJ 07102}
\email{liguo@rutgers.edu}
\date{\today}
\begin{abstract}
The study of Reynolds algebras has its origin in the well-known work of O. Reynolds on fluid dynamics in 1895 and has since found broad applications. It also has close relationship with important linear operators such as algebra endomorphisms, derivations and Rota-Baxter operators. Many years ago G.~Birkhoff suggested an algebraic study of Reynolds operators, including the corresponding free algebras. We carry out such a study in this paper. We first provide examples and properties of Reynolds operators, including a multi-variant generalization of the Reynolds identity. We then construct the free Reynolds algebra on a set. For this purpose, we identify a set of bracketed words called Reynolds words which serves as the linear basis of the free Reynolds algebra. A combinatorial interpretation of Reynolds words is given in terms of rooted trees without super crowns. The closure of the Reynolds words under concatenation gives the algebra structure on the space spanned by Reynolds words. Then a linear operator is defined on this algebra such that the Reynolds identity and the desired universal property are satisfied.
\end{abstract}

\subjclass[2010]{
16W99, 
17A36   
16S10 
05C05, 
76D99   
}

\keywords{Reynolds operator, Reynolds algebra, averaging operator, free object, bracketed word, rooted trees}

\maketitle

\tableofcontents

\vspace{-1.5cm}

\allowdisplaybreaks

\section{Introduction}
This paper starts a systematic algebraic study of Reynolds algebras, focusing on the construction of their free objects by bracketed words and rooted trees. In doing do, we solved a problem posted by G. Birkhoff in 1961~\cite{Bi}.

\subsection{History of Reynolds operators}

In his celebrated study~\cite{Re} of fluctuation theory in fluid dynamics in 1895, O. Reynolds was led to a linear operator $P$ on an algebra $R$ of functions satisfying the operator identity
\begin{equation}
P(uv) = P(u)P(v)+ P((u-P(u))( v-P(v) )) \tforall u, v \in R,
\mlabel{eq:rey}
\end{equation}
which is also known as the approximation relation.
In its early applications in fluid dynamics, in particular in the Reynolds-averaged Navier-Stokes equations modeling turbulent flows, a Reynolds operator played the role of taking average over a time interval.
The term {\bf Reynolds operator} was coined by J. Kamp\'{e} de F\'{e}riet who pursued the study of the operator as a mathematical subject in general for an extended period of time~\cite{Ka1,Ka2}. More commonly used equivalent form of Eq.~\eqref{eq:rey} is the {\bf Reynolds identity}
\begin{equation}
P(u)P(v) = P(uP(v))+P(P(u)v)-P(P(u) P(v)) \tforall u, v\in R,
\mlabel{eq:rey2}
\end{equation}
which was derived by M.-L. Dubreil-Jacotin~\mcite{DJ}.

An operator closely related to the Reynolds operator is the {\bf averaging operator}, defined to be a linear operator $P$ on a $\bfk$-algebra $R$ that satisfies the {\bf averaging identity}
\begin{equation}
P(u P(v)) = P(P(u)v) = P(u)P(v) \tforall u, v\in R.
\mlabel{eq:ave}
\end{equation}
The operator was first studied by G. Birkhoff and J. Kamp\'{e} de F\'{e}riet. It is well known that an idempotent averaging operator is a Reynolds operator.
There is a quite large literature on averaging operators related to algebra, analysis, combinatorics, geometry and operads~\cite{Bo,Br,Cao,GuK,Ki,PBGN,PG,Ro1,Tr}.

As realized many years later, the study of Reynolds and averaging operators in the 1930s was closely related to the probability theory that Kolmogorov was developing at the same time, in particular to conditional expectations.
Large thanks to the contributions of G.~Birkhoff, J.~B.~Miller and G.-C.~Rota and their coauthors, there was an extensive literature on Reynolds operators in the 1950s and 1960s, for example on $f$-algebras, algebras of finitely valued functions, operators on function spaces and generalized connection with conditional expectation~\cite{FM,Mi1,Mi2}.
See the survey articles of G.~Birkhoff and G.-C.~Rota~\cite{Bi,Ro2}.

Recently, Reynolds operators were defined for rational $G$-modules where $G$ is a linearly reductive group, and dual functors of $G$-modules~\cite{AS,CW,CHK}, motivated by geometric invariant theory~\cite{MFK} and lattice ordered algebras~\mcite{BBT}.

\subsection{Motivation of our approach}

According to the early practitioners of Reynolds operators such as Reynolds, Rota and Miller~\mcite{Mi2,Re,Ro2}, the Reynolds operator first arose because of the difficulty of producing averaging operators that commuted with differentiation in Euclidean spaces; it can be regarded as an approximation of the algebra endomorphism, a formal translated inverse of the derivation and an infinitesimal infinitesimal analog of the Rota-Baxter operator. See Section~\mref{ss:defex} for some details.
Another connection that is especially pertinent to us is that Reynolds operators and their generalizations provides a suitable abstraction for the Volterra integral operators, as shown in Examples~\mref{ex:ie1} and~\mref{ex:ie2}~\mcite{GGL,Ze}.

Thus with its long history, broad applications and intimate connections with classical linear operators, it is important to give a systematic study of the Reynolds operator in the algebraic context. In fact, in his classical survey~\mcite{Bi}, R. Birkhoff suggested the following problem, which we will refer as {\bf Birkhoff's Question}.

\begin{question} $($\cite{Bi}$)$
It would be interesting to study the algebraic implications of this condition~$[$referring to Eq.~\eqref{eq:rey2}$]$ $($including the ``free algebra" of linear operators satisfying Eq.~\eqref{eq:rey2}$)$.
\label{qu:birk}
\end{question}
Such an algebraic study of Reynolds operators, especially the resolution of Birkhoff's Question on free Reynolds algebras, is what the current paper intends to carry out.

Combinatorial properties of an algebraic structure are often reflected by its free objects. This is the case for the well-known constructions of free semigroups and free (commutative and noncommutative) algebras as words and (commutative and noncommutative) polynomials respectively. Less well-known is the case of free Lie algebras~\mcite{Reu}. For algebras with linear operators, free Rota-Baxter algebras~\mcite{EG,Gub,GS}, free averaging algebras~\mcite{PG} and free operated algebras~\mcite{Gop,Gub} have their combinatorial interpretations as quasi-symmetric functions, rooted trees and Motzkin paths, and have their generating functions related to Catalan numbers and Schr\"oder numbers.

For Reynolds algebras, R.~Birkhoff already noticed the significance of their free objects in his Question~\ref{qu:birk}.
Even though the existence of free Reynolds algebra comes from general results of universal algebra to which R.~Birkhoff is a main contributor, their explicit construction is still not known. This is in sharp contrast with not only the classical algebraic structures aforementioned, but also the other well-known algebraic structures equipped with linear operators, such differential algebras~\mcite{Kol}, Rota-Baxter algebras~\mcite{Ca,GK,Ro} and the closely related averaging algebras~\mcite{PG}.
In particular, free differential algebras were obtained by Ritt and Kolchin in the first half of the last century~\mcite{Kol}; free commutative Rota-Baxter algebras were first constructed by Rota and then Cartier around 1970~\mcite{Ca,Ro}.

To understand the challenges in constructing free Reynolds algebras and the motivation of our study, let us compare the Reynolds algebra with these other algebras.

An algebra $R$ together with a linear operator $P:R\longrightarrow R$ is called
\begin{enumerate}
\item a {\bf differential algebra} if $P$ satisfies the Leibniz role
\begin{equation}
P(uv)=P(u)v+uP(v) \tforall u, v\in R;
\mlabel{eq:diff}
\end{equation}
\item
a {\bf Rota-Baxter algebra} of weight $\lambda$, where $\lambda$ is a fixed scalar, if $P$ satisfies the generalized integration by parts
\begin{equation}
P(u)P(v)=P(uP(v))+P(P(u)v)+\lambda P(uv) \tforall u, v\in R;
\mlabel{eq:rb}
\end{equation}
\item
an {\bf averaging algebra} if $P$ satisfies Eq.~(\mref{eq:ave}).
\end{enumerate}
There is a more general class of Rota-Baxter type algebras which also include the Nijenhuis algebras and Leroux's TD algebras~\mcite{ZGGS}.

The idea behind the constructions of free objects in these algebraic structures is the fact that the relation defining each of the structures gives rise to a convergent rewriting system that reduces the left hand side to the right hand side~\mcite{LG}. Despite its close similarity with the Rota-Baxter operator, this reduction does not work for Reynolds algebras, since the intended leading term $P(u)P(v)$ on the left hand side of Eq.~(\ref{eq:rey2}) also appears on the right hand side, resulting in a divergent rewriting system unless some completeness conditions are imposed (see Proposition~\mref{pp:series}). Our first strategy is to extract a different rewriting rule from the Reynolds identity in Eq.~(\mref{eq:rey2}), giving rise to a terminating rewriting system. Even with that issue resolved, it needs careful analysis of the reduction, including multi-variant generalizations of the Reynolds identity, in order to arrive at a set of reduced words in the set of bracketed words, and to verify that the space spanned by these reduced words has a Reynolds algebra structure and satisfies the desired universal property.

\subsection{Outline of the paper}

Pursuing this approach, the layout of the paper is as follows. In Section~\mref{sec:prel}, we first provide examples of Reynolds operators from integration and derivation, followed by some basic properties of Reynolds operators, including an interpretation (Proposition~\mref{pp:series}) of Rota's remark that the Reynolds operator should be an infinitesimal analog of the Rota-Baxter operator. We then give a multi-variant generalization (Proposition~\mref{prop:R}) of the Reynolds identity in Eq.~\eqref{eq:rey2}, which other than being interesting on its own, plays an essential role on the construction of the free Reynolds algebras in the next section.

The purpose of Section~\mref{sec:free} is to construct free Reynolds algebras as suggested in Reynolds's Question~\ref{qu:birk}. We first recall the construction of free operated algebras on a set via bracketed words with the set as the alphabet. This free operated algebra serves as the space from which to derive a free Reynolds algebra by taking the quotient modulo the operated ideal generated by the Reynolds identity. In order to give an explicitly defined basis of this quotient, we identify a canonical subset of the bracketed words, called the Reynolds words, as a complete set of representatives. The combinatorial meaning of Reynolds words in terms of trees is illustrated. The rest of the section is devoted to the verification that the space spanned by Reynolds words has the algebraic structure and the universal property of a free Reynolds algebra, culminated in the main result of the paper, Theorem~\mref{thm:main}.

\smallskip

{\bf Convention. } Throughout this paper, let $\bfk$ be a unitary commutative ring which will be the base ring of all  modules, algebras, as well as linear maps. Unless otherwise specified, an algebra is a unitary associative $\bfk$-algebra. Denote by $M(X)$ (resp. $S(X)$) the free monoid (resp. semigroup) generated by $X$. For any set $Y$, denote by $\bfk Y$ the free $\bfk$-module with basis $Y$.

\section{Properties and examples}
\mlabel{sec:prel}

In this section, we give first give examples of Reynolds operators with motivations from integration and differentiation. We then give a weighted generalization and properties of Reynolds algebras, including their replicating property and role as infinitesimal or deformation of Rota-Baxter algebras. We also show that the Reynolds identity~\eqref{eq:rey2} has natural multi-variant generalizations.

\subsection{Definitions and examples}
\mlabel{ss:defex}
From now on, we refer Eq.~\eqref{eq:rey2} as the Reynolds identity. Basic concepts for algebras and Rota-Baxter algebras~\mcite{Gub} can be similarly defined for Reynolds algebras.
\begin{defn}
A {\bf Reynolds subalgebra} (resp. {\bf Reynolds ideal}) of a Reynolds algebra $(R,P)$ is a subalgebra (resp. an ideal ) $I$ of $R$ such that $P(I)\subseteq I$.
A {\bf Reynolds algebra homomorphism} $f:(R_{1},P_{1})\rightarrow(R_{2},P_{2})$ between two Reynolds algebras $(R_{1},P_{1})$ and $(R_{2},P_{2})$ is an algebra homomorphism such that $fP_{1}=P_{2}f$.
\end{defn}

Let $\calr\cala$ denote the category of Reynolds algebras with Reynolds algebra homomorphisms. In Section~\mref{sec:free}, we will construct the free objects of $\calr\cala$ by using the Reynolds words. We first give some general properties of Reynolds algebras in this section.

There is a close relationship between Reynolds operators and averaging operators as defined in Eq.~\eqref{eq:ave}. In fact, an idempotent linear operator is Reynolds if and only if it is averaging, a fact that is known to Miller and Rota~\cite{BR,Mi2}.

Indeed most known examples of Reynolds operators are idempotent and thus are given in the form of idempotent averaging operators~\mcite{PG,Ro1}; while most known non-idempotent Reynolds operators are the classical ones from analysis, such as the following ones. We will later provide some algebraically defined examples.

\begin{exam}~\cite{Ro2} Let $T^{t}$ be a one parameter semigroup of measure preserving transformations of a measure space $\left(S, \sum, m\right)$, and define $V^{t}f(s) = f(T^{t}s), s\in S$. The operator
$$R(f)(s): = \int^{\infty}_{0}e^{-t}V^{t}f(s) dt$$
is a Reynolds operator on the algebra $L_{\infty}\left(S, \sum, m\right)$ of bounded measurable functions on $(S, \sum, m)$.
\end{exam}

\begin{exam} ~\cite{BR}
Fix $-\infty<a< b\leq \infty$. Let $\mathrm{Cont}([a,b))$ be the $\RR$-algebra of real valued continuous functions on $[a,b)$ vanishing at $b$. Then $P(f)(x):= e^{-x}\int_{a}^{x}f(t)e^{t} dt$ is a Reynolds operator. It is not an averaging operator.
\mlabel{ex:ie1}
\end{exam}

\begin{remark}
This operator $P$ can be regarded as being defined by the convolution product $P(f):=e^{-x}\ast f$ where $(g\ast f)(x):=\int_{a}^0 g(t)f(x-t)\,dt$.
\end{remark}

The integration in the above example is a special case of Volterra (integral) operators~\mcite{Ze} with a kernel $K(x,t)$:
$$P_K(f)(x):=\int_a^x K(x,t)f(t)dt$$
defined on a suitable algebra of functions. So the above example is the case when $K(x,t)=e^{t-x}$. The following is another example. See~\mcite{GGL} for a general approach.

\begin{exam} \mlabel{ex:ie2}
Let $K(x,t)=x^{-1}$, then the linear operator $P=P_K$ on $C(\RR)$ gives
\[P(f)(x):=x^{-1}I(f)(x),\,f\in C(\RR),\]
where $I(f)(x)=\int_0^xf(t)dt$. Then for $f, g\in C(\RR)$, the difference of the two sides of Eq.~\eqref{eq:rey2} is
\begin{align*}
P(f)(x)&P(g)(x)-P(P(f)g)(x)-P(fP(g))(x)-P(P(f)P(g))(x)\\
&= \frac{1}{x^2}I(f)(x)I(g)(x) -\frac{1}{x}I\bigg(\frac{1}{t}I(f)g\bigg)(x) -\frac{1}{x}I\bigg(\frac{1}{t}fI(g)\bigg)(x) -\frac{1}{x}I\bigg(\frac{1}{t^2}I(f)I(g)\bigg).
\end{align*}

Denote the last expression by $F(x)$ and let $G(x):=xF(x)$. Then a direct check shows that $G'(x)=0$, implying that $G(x)$ is a constant. But $G(0)=0$ since $I(h)(x)=0$ for any $g\in C(\RR)$. Hence $G(x)$ is identically zero. Therefore $F(x)$ is also identically zero. This verifies the Reynolds identity for the operator $P$.
\end{exam}

We now provide some algebraically defined Reynolds operators that are not idempotent, starting with a simple one extracted from the previous example, with an algebraic proof.

\begin{exam}
In the polynomial algebra $\bfk[x]$, define a linear map by $$P:\bfk[x]\to \bfk[x],x^n\mapsto P(x^n):=\frac{1}{n+1}x^n, n\geq 0.$$
Then we have
$$
P(P(x^n)P(x^m)) =P\left(\frac{1}{(n+1)(m+1)}n^{n+m}\right)
=\frac{1}{(n+1)(m+1)(m+n+1)}x^{n+m}
$$
and
\begin{align*}
&P\left(P(f)g \right) + P\left(fP(g)\right) - P(f)P(g)\\
&=\left(\frac{1}{(n+1)(n+m+1)}+\frac{1}{(m+1)(n+m+1)}- \frac{1}{(n+1)(m+1)}\right)x^{n+m}\\
&=\frac{1}{(n+1)(m+1)(m+n+1)}x^{n+m}.
\end{align*}
Thus
\begin{equation*}
P\left(P(f)P(g)\right) = P\left(P(f)g \right) + P\left(fP(g)\right) - P(f)P(g) \tforall f, g\in \bfk[x]
\end{equation*}
and $P$ is a Reynolds operator on $\bfk[x]$.
\end{exam}

According to Rota~\cite{Ro2}, if $P$ is a Reynolds operator and if $P^{-1}$ exists, then $P^{-1}-\id$ is a derivation, where $\id$ is the identity operator. This leads to the natural question of deriving a Reynolds operator from a derivation. Our next result is to confirm this.

We first give the following combinatorial formula as a preparation.
\begin{lemma}
Let $p\geq1,1\leq r\leq q$ be the integer numbers. Then
\begin{equation}
\binc{p+r}{r-1}+\sum^{q}_{j=r}\binc{p+j}{j}=\binc{p+q+1}{p+1}.
\mlabel{eq:Icom}
\end{equation}
In special case when $r=1$, we obtain
\begin{equation}
\sum^{q}_{j=0}\binc{p+j}{j}=\binc{p+q+1}{p+1}.
\mlabel{eq:Icom'}
\end{equation}
\end{lemma}
\begin{proof}
We verify Eq.~(\mref{eq:Icom}) by induction on $q-r\geq 0$. When $q-r=0$, then $q=r$. Thus
\begin{align*}
\binc{p+r}{r-1}+\sum^{q}_{j=r}\binc{p+j}{j}=\binc{p+q}{q-1}+\binc{p+q}{q}
=\binc{p+q+1}{q+1}=\binc{p+q+1}{p+1}.
\end{align*}

Assume that Eq.~(\mref{eq:Icom}) holds for $q-r=k\geq0$ and consider the case when $q-r=k+1$. Then $q-1-r=k$. Thus
\begin{align*}
\binc{p+r}{r-1}+\sum^{q}_{j=r}\binc{p+j}{j}
&=\binc{p+r}{r-1}+\sum^{q-1}_{j=r}\binc{p+j}{j}+\binc{p+q}{p}\\
&=\binc{p+q}{q-1}+\binc{p+q}{p}\quad(\text{by the induction hypothesis})\\
&=\binc{p+q}{p-1}+\binc{p+q}{p}\\
&=\binc{p+q+1}{p}.
\end{align*}
\end{proof}

\begin{prop}
Let ($R, d$) be a differential algebra, as defined in Eq.~(\mref{eq:diff}), such that the series $\sum\limits_{n=0}^\infty d^n(r)$ is convergent for all $r\in R$. Then the operator
$P:= \sum\limits_{n=0}^\infty(-1)^{n}d^{n}$ defines a Reynolds operator on $R$.
This is the case when $d$ is locally nilpotent on $R$ in the sense that, for any $r\in R$, there is $N:=N_{r}\in \NN$, such that $d^{N}(r)=0$.
\mlabel{prop:diffrey}
\end{prop}

\begin{proof}
Evidently, if $d$ is locally nilpotent on $R$, then the series $\sum\limits_{n=0}^\infty d^n(r)$ is a finite sum and hence convergent. Thus we only need to prove the first statement.

For any $u,v \in R$, we have
\begin{equation}
P(u)P(v)= \sum\limits^{}_{\mbox{\tiny$\begin{array}{c}m\geq0, n\geq0\end{array}$}}(-1)^{m+n}d^{m}(u)d^{n}(v).
\mlabel{eq:dr1}
\end{equation}
Applying the generalized Leibniz rule
$$
d^k(uv)=\sum_{i+j=k}\binc{k}{i} d^i(u)d^j(v),$$
we have
\begin{align*}
P(P(u)P(v))&= \sum_{m, n, k\geq 0}
\sum_{i+j=k} (-1)^{m+n+k}{{k}\choose{i}}d^{m+i}(u)d^{n+j}(v)\\
&= \sum_{m, n, i, j\geq 0} (-1)^{m+n+i+j} \binc{i+j}{i} d^{m+i}(u)d^{n+j}(v)\\
&= \sum_{p, q\geq 0} \sum_{m+i=p,n+j=q} (-1)^{p+q} \binc{i+j}{i} d^p(u)d^q(v)\\
&=\sum_{p, q\geq 0} (-1)^{p+q}\left(\sum_{i=0}^p
\sum_{j=0}^q {{i+j}\choose{i}}\right)d^{p}(u)d^{q}(v).
\end{align*}
Similarly, we obtain
\begin{align*}
P(P(u)v)&= \sum\limits^{}_{\mbox{\tiny$\begin{array}{c}m\geq0, k\geq0\end{array}$}}
\sum\limits^{}_{\mbox{\tiny$\begin{array}{c}i+j=k\end{array}$}}(-1)^{m+k}{{k}\choose{i}}d^{m+i}(u)d^{j}(v)\\
&=\sum\limits^{}_{\mbox{\tiny$\begin{array}{c}l\geq0\end{array}$}}(-1)^{l}\sum\limits^{}_{\mbox{\tiny$\begin{array}{c}p+j=l\end{array}$}}
\sum\limits^{}_{\mbox{\tiny$\begin{array}{c}m+i=p\end{array}$}}{{i+j}\choose{i}}d^{p}(u)d^{j}(v)\\
&=\sum\limits^{}_{\mbox{\tiny$\begin{array}{c}p\geq0, j\geq0\end{array}$}}(-1)^{p+j}\sum\limits^{p}_{\mbox{\tiny$\begin{array}{c}i=0\end{array}$}}
{{i+j}\choose{i}}d^{p}(u)d^{j}(v)
\end{align*}
and
\begin{align*}
P(uP(v))&=
\sum\limits^{}_{\mbox{\tiny$\begin{array}{c}q\geq0, i\geq0\end{array}$}}(-1)^{q+i}\sum\limits^{q}_{\mbox{\tiny$\begin{array}{c}j=0\end{array}$}}
{{i+j}\choose{j}}d^{i}(u)d^{q}(v).
\end{align*}
Substituting these equations into the Reynolds identity (\ref{eq:rey2}) and comparing the coefficients of $d^m(u)d^n(v)$, we see that we only need to verify the equation
\begin{equation}
\sum\limits^{p}_{\mbox{\tiny$\begin{array}{c}i=0\end{array}$}}\sum\limits^{q}_{\mbox{\tiny$\begin{array}{c}j=0\end{array}$}}{{i+j}\choose{i}}
=\sum\limits^{p}_{\mbox{\tiny$\begin{array}{c}i=0\end{array}$}}
{{i+q}\choose{i}} + \sum\limits^{q}_{\mbox{\tiny$\begin{array}{c}j=0\end{array}$}}
{{p+j}\choose{j}}-1
\mlabel{eq:indeq}
\end{equation}
for which we proceed by applying the induction on $k:=p+q$.

In the case when $p+q=0$, namely $p=q=0$, the equation is checked directly. Let $n\geq 0$. Assume that Eq.~(\mref{eq:indeq}) holds for all $p, q\geq 0$ with $p+q= n$. Consider a pair $p, q \geq 0$ with $p+q=n+1$. So $p+q\geq 1$. By the symmetry of Eq.~(\ref{eq:indeq}) in $p$ and $q$, we can assume $p\geq 1$. Then $p-1+q=n$.
Applying the induction hypothesis, the left hand side of Eq.~(\mref{eq:indeq}) becomes

\begin{align*}
\sum\limits^{p}_{\mbox{\tiny$\begin{array}{c}i=0\end{array}$}}
\sum\limits^{q}_{\mbox{\tiny$\begin{array}{c}j=0\end{array}$}}{{i+j}\choose{i}}
&=\sum\limits^{p-1}_{\mbox{\tiny$\begin{array}{c}i=0\end{array}$}}
\sum\limits^{q}_{\mbox{\tiny$\begin{array}{c}j=0\end{array}$}}{{i+j}\choose{i}}+ \sum\limits^{q}_{\mbox{\tiny$\begin{array}{c}j=0\end{array}$}}{{p+j}\choose{j}}\\
&=\sum\limits^{p-1}_{\mbox{\tiny$\begin{array}{c}i=0\end{array}$}}
{{i+q}\choose{i}} + \sum\limits^{q}_{\mbox{\tiny$\begin{array}{c}j=0\end{array}$}}
{{p-1+j}\choose{j}}-1+ \sum\limits^{q}_{\mbox{\tiny$\begin{array}{c}j=0\end{array}$}}{{p+j}\choose{j}}
\end{align*}
while the right hand side of Eq.~(\mref{eq:indeq}) becomes
$$
\sum\limits^{p}_{\mbox{\tiny$\begin{array}{c}i=0\end{array}$}}
{{i+q}\choose{i}} + \sum\limits^{q}_{\mbox{\tiny$\begin{array}{c}j=0\end{array}$}}
{{p+j}\choose{j}}-1
= \sum\limits^{p-1}_{\mbox{\tiny$\begin{array}{c}i=0\end{array}$}}
{{i+q}\choose{i}} + \sum\limits^{q}_{\mbox{\tiny$\begin{array}{c}j=0\end{array}$}}
{{p+j}\choose{j}}-1+{{p+q}\choose{p+1}}.
$$
Therefore, we only need to prove $$\sum\limits^{q}_{\mbox{\tiny$\begin{array}{c}j=0\end{array}$}}
{{p-1+j}\choose{j}}={{p+q}\choose{p+1}}.$$
That is,
$$\sum\limits^{q}_{\mbox{\tiny$\begin{array}{c}j=0\end{array}$}}
{{p+j}\choose{j}}={{p+q+1}\choose{p+1}},$$
which is simply Eq.~\eqref{eq:Icom'}. This completes the induction.
\end{proof}

As an application of Proposition~\mref{prop:diffrey}, we construct Reynolds operators on the polynomial algebra.

\begin{exam}
Let $\RR[x]$ be the polynomial algebra with the standard derivation $d:=\frac{d}{dx}$. Then $d$ is locally nilpotent since $d^{n+1}(f)=0$ if $f$ has degree $n$. Hence by Proposition~\mref{prop:diffrey},
$$P:= \sum\limits_{n=0}^\infty(-1)^{n}d^{n}$$
is a Reynolds operator on $\RR[x]$.
\end{exam}

\subsection{Generalizations and basic properties}

For broader applications, we generalize the concept of the Reynolds operator as follows.
\begin{defn}
Let $\bfk$ be a unitary commutative ring and $\lambda$ a given element of $\bfk$. A {\bf Reynolds algebra $R$ of weight $\lambda$} is a pair $(R,P)$ consisting of an algebra $R$ and a linear operator $P: R\rightarrow R$ that satisfies the following {\bf Reynolds identity of weight $\lambda$}:
\begin{equation}
P(u)P(v) = P(uP(v))+P(P(u)v)+\lambda P(P(u) P(v))\tforall u,v \in R.
\mlabel{eq:reywt}
\end{equation}
\end{defn}
In this sense, the Reynolds operator is the special case when the weight is $-1$ and the Rota-Baxter operator of weight zero is the special case of a weighted Reynolds operator when the weight is zero. It is easy to check that if $P$ is a Reynolds operator of weight $\lambda$, then $\lambda P$ is a Reynolds operator of weight $1$ and $-P$ is a Reynolds operator of weight $-\lambda$.

As in the case of Rota-Baxter algebras~\mcite{Gub}, a Reynolds algebra structure can replicate itself as follows.

\begin{theorem}
Let $(R,P)$ be a Reynolds algebra of weight $\lambda$. Define a multiplication $\star$ on $R$ by
\begin{equation}
u\star v := uP(v) + P(u)v + \lambda P(u)P(v)\tforall u,v\in R.
\mlabel{eq:ndef}
\end{equation}
Then
\begin{enumerate}
\item
$P(u)P(v) = P(u\star v)$.
\mlabel{it:repa}
\smallskip
\item
$(R, \star)$ is an algebra.
\mlabel{it:repb}
\item
$(R, \star, P)$ is a Reynolds algebra of weight $\lambda$.
\mlabel{it:repc}
\item
$P$ is a Reynolds algebra homomorphism from $(R, \star, P)$ to $(R, \cdot, P)$. Here $\cdot$ denotes the original multiplication on $R$.
\mlabel{it:repd}
\end{enumerate}
\end{theorem}

\begin{proof}
(\mref{it:repa}). It follows directly from the Reynolds identity in Eq.~(\mref{eq:rey2}).
\smallskip

\noindent
(\mref{it:repb}).
We just need to verify the associativity of $\star$, that is, for $u,v,w\in R$, the following holds
\begin{equation}
(u \star v) \star w = u \star (v \star w).
\mlabel{eq:repprod}
\end{equation}
Indeed,
\begin{align*}
(u \star v) \star w
&= (u \star v)P(w) + P(u \star v)w + \lambda P(u \star v)P(w)\quad (\mbox{by Eq.}~(\mref{eq:ndef}))\\
&= (u \star v)P(w) + (P(u)P(v))w + \lambda (P(u)P(v))P(w)\quad (\mbox{by Item}~(\mref{it:repa}))\\
&= (u P(v))P(w) + (P(u)v)P(w) + 2\lambda (P(u)P(v))P(w) + (P(u)P(v))w \quad\text{(by Eq.~(\mref{eq:ndef})).}
\end{align*}
Similarly,
$$u \star (v \star w)= u(P(v)P(w)) + P(u)(v P(w)) + 2\lambda P(u)(P(v)P(w)) + P(u)(P(v)w).$$
Hence Eq.~(\mref{eq:repprod}) follows from the associativity of $(R, \cdot)$.

(\mref{it:repc}). For $u,v\in R$, we obtain
\begin{align*}
P(u)\star P(v)
&= P(u)P^{2}(v) + P^{2}(u)P(v) + \lambda P^{2}(u)P^{2}(v)\quad(\mbox{by Eq.}~(\mref{eq:ndef}))\\
&= P(u\star P(v)) + P(P(u)\star v) + \lambda P(P(u)\star P(v)),\quad(\mbox{by Item}~(\mref{it:repa}))
\end{align*}
as needed.

(\mref{it:repd}). By Item~(\mref{it:repa}), $P$ is an algebra homomorphism. Furthermore, $P$ commutes with itself. So it is a Reynolds algebra homomorphism.
\end{proof}
Rota~\mcite{Ro2} suggested that the Reynolds operator is an infinitesimal analog of the Rota-Baxter operator. We give the following interpretation.

Let $(R,P)$ be a Reynolds algebra and $u, v$ two arbitrary elements in $R$. Denote $u\ast v:= uP(v)+P(u)v$. Then $P$ is a Rota-Baxter operator of weight zero means that $P(u)P(v)=P(u\ast v)$. On the other hand, using this notation, the Reynolds equation of weight $\lambda$ in Eq.~(\mref{eq:reywt}) becomes
$$ P(u)P(v)=P(u\ast v)+\lambda P(P(u)P(v)).$$
By repeatedly applying Eq.~(\mref{eq:reywt}) to the $P(u)P(v)$ on its right hand side of this equation, the right hand side yields
\begin{equation}
P(u)P(v)=\sum_{n=0}^k \lambda^n P^{n+1}(u\ast v) +\lambda^{k+1}P^{n+1}(P(u)P(v)),
\mlabel{eq:inf0}
\end{equation}
and eventually
\begin{equation}
P(u)P(v)=\sum_{n=0}^\infty \lambda ^n P^{n+1}(u\ast v)
\mlabel{eq:inf}
\end{equation}
formally. This can be made precise if the right hand side makes sense as we now show. See~\mcite{Gub} for example for background on completions and inverse limits.

\begin{prop}
Let $R\cong \invlim R/I_n$ be a complete filtered algebra given by the decreasing sequence of ideals $I_n$ of $R$. If $P$ is a Reynolds operator on $R$ such that $P(I_n)\subseteq I_{n+1}$, then Eq.~(\mref{eq:inf}) holds.
\mlabel{pp:series}
\end{prop}

\begin{proof}
For each given $k\geq 0$, by Eq.~(\mref{eq:inf0}), the difference
$P(u)P(v)=\sum_{n=0}^k \lambda^n P^{n+1}(u\ast v) $ lies in $P^{n+1}(R)$. Since $R$ is complete with respect to the metric defined by $I_n$ and $P(I_n)\subseteq I_{n+1}, n\geq 0,$ we have
$$\bigcap_{k\geq 0}P^{k+1}(R)=0.$$
This proves the proposition.
\end{proof}

As noted in the introduction, the format of the Reynolds identity is similar to the identity of the Rota-Baxter operator:
$$P(u)P(v)=P(uP(v))+P(P(u)v)+\lambda P(uv) \tforall u, v\in R$$
or more generally the Rota-Baxter type operators listed in~\mcite{ZGGS}. The key difference is that the left hand side also appears in the right hand side of the equation, leading to an infinite loop if the equation is used a rewriting role replacing the left hand side by the right hand side. This poses challenges in constructing free Reynolds algebras. To overcome this difficulty, we  use the following alternative form of the Reynolds identity.
\begin{equation}
P(P(u) P(v)) = P(uP(v))+ P(P(u)v)- P(u)P(v) \tforall u,v \in R
\mlabel{eq:rey3}
\end{equation}
by keeping the term with the highest total order of $P$ to the left hand side. As it turns out, the reduction role given by this equation is still not enough for the general purpose. Instead, there are multi-variant Reynolds identities as shown below that one also needs to take into consideration in the reduction.

\begin{prop}
Let $m\geq2$ be a positive integer and $(R,P)$ a Reynolds algebra of weight $-1$. Then
\begin{equation}
(m-1)P\bigg(\prod_{i=1}^{m}P(u_{i})\bigg)=\sum_{i=1}^{m}P(P(u_{1})\cdots P(u_{i-1})u_{i} P(u_{i+1})\cdots P(u_{m}))-\prod_{i=1}^{m}P(u_{i}),
\mlabel{eq:R}
\end{equation}
where $u_{i},1\leq i\leq m,$ are in $R$.
\mlabel{prop:R}
\end{prop}
Evidently Eq.~\eqref{eq:rey3} is the special case of Eq.~\eqref{eq:R} when $m=2$.
Roughly speaking, the equation can be regarded as a generalized Leibniz rule that expresses the left hand side as a sum in which each term is obtained from the left by omitting a $P$ from one of the locations. Here are some examples.
\begin{exam}
Let $x, y, z$ be three elements in a Reynolds algebra $(R,P)$ of weight $-1$. Then applying Eq.~\eqref{eq:R} gives
\begin{align*}
P(P( x) P( y)P( z))&=\frac{1}{2}P( xP( y)P( z) )+\frac{1}{2}P(P( x) yP( z))+\frac{1}{2}P(P( x)P( y) z)-\frac{1}{2}P( x)P( y)P( z).\\
P(P^{2}( x) P( y)P( z))&=\frac{1}{2}P( P(x)P( y)P( z) )+\frac{1}{2}P(P^2( x) yP( z))+\frac{1}{2}P(P^2( x)P( y) z)-\frac{1}{2}P^2( x)P( y)P( z)\\
&=\frac{1}{4}P( xP( y)P( z) )+\frac{1}{4}P(P( x) yP( z))+\frac{1}{4}P(P( x)P( y) z)-\frac{1}{4}P( x)P( y)P( z)\\
&~~+\frac{1}{2}P(P^2( x) yP( z))+\frac{1}{2}P(P^2( x)P( y) z)-\frac{1}{2}P^2( x)P( y)P( z).
\end{align*}
\mlabel{ex:R}
\end{exam}
\begin{proof} {\bf (of Proposition~\mref{prop:R})}
To simplify notations, we first use the abbreviation
\begin{equation*}
u^{*}_{i}=u^*_{m,i}:=P(u_{1})\cdots P(u_{i-1})u_{i} P(u_{i+1})\cdots P(u_{m}), \quad i=1,\cdots,m, \ m\geq 2. \end{equation*}
Then Eq.~\eqref{eq:R} becomes
\begin{equation}
(m-1)P\bigg(\prod_{i=1}^{m}P(u_{i})\bigg)=\sum_{i=1}^{m}P(u^*_i)-\prod_{i=1}^{m}P(u_{i}). \mlabel{eq:R2}
\end{equation}

We now prove Eq.~(\mref{eq:R}) by induction on $m\geq 2$ with the case when $m=2$ coming from Eq.~(\mref{eq:rey3}).
For $k\geq2$, assume that Eq.~(\mref{eq:R}) holds for $m\leq k$ and consider the case when $m=k+1$. By the induction hypothesis, the following equation holds
\begin{equation}
P\bigg(\prod_{i=1}^{k}P(u_{i})\bigg)
=\frac{1}{k-1}\left(\sum_{i=1}^{k}P(u^{*}_{i})-\prod_{i=1}^{k}P(u_{i})\right).
\mlabel{eq:H0}
\end{equation}
Furthermore, by Eq.~(\mref{eq:rey3}), we obtain
\begin{align*}
P\bigg(P\bigg(\prod_{i=1}^{k}P(u_{i})\bigg)P( u_{k+1})\bigg)
=P\bigg(\prod_{i=1}^{k+1}P(u_{i})\bigg)+P\bigg(P\bigg(\prod_{i=1}^{k}P(u_{i})\bigg)u_{k+1}\bigg)
-P\bigg(\prod_{i=1}^{k}P(u_{i})\bigg)P( u_{k+1}),
\end{align*}
from which we obtain
$$
P\bigg(\prod_{i=1}^{k+1}P(u_{i})\bigg)
=P\bigg(P\bigg(\prod_{i=1}^{k}P(u_{i})\bigg)P( u_{k+1})\bigg)+P\bigg(\prod_{i=1}^{k}P(u_{i})\bigg)P( u_{k+1})-P\bigg(P\bigg(\prod_{i=1}^{k}P(u_{i})\bigg)u_{k+1}\bigg).
$$
Applying the induction hypothesis in Eq.~\eqref{eq:H0} to the three terms on the right hand side, we obtain
\begin{align*}
P\bigg(\prod_{i=1}^{k+1}P(u_{i})\bigg)
&=\frac{1}{k-1}\left(\sum_{i=1}^{k}P(P(u^{*}_{i})P(u_{k+1}))-P\bigg(\prod_{i=1}^{k+1}P(u_{i})\bigg)\right)
+\frac{1}{k-1}\left(\sum_{i=1}^{k}P(u^{*}_{i})P( u_{k+1})-\prod_{i=1}^{k+1}P(u_{i})\right)\\
&~~+\frac{1}{k-1}\left(-\sum_{i=1}^{k}P(P(u^{*}_{i})u_{k+1})+P\bigg(\prod_{i=1}^{k}P(u_{i}) u_{k+1}\bigg)\right).
\end{align*}
Applying Eq.~\eqref{eq:rey3} again, we obtain
\begin{align*}
P\bigg(\prod_{i=1}^{k+1}P(u_{i})\bigg)
&=\frac{1}{k-1}\sum_{i=1}^{k}P(u^{*}_{i}P( u_{k+1}))
+\frac{1}{k-1}\sum_{i=1}^{k}P(P(u^{*}_{i})u_{k+1})
-\frac{1}{k-1}\sum_{i=1}^{k}P(u^{*}_{i})P( u_{k+1})\\
&~~-\frac{1}{k-1}P\bigg(\prod_{i=1}^{k+1}P(u_{i})\bigg)
+\frac{1}{k-1}\sum_{i=1}^{k}P(u^{*}_{i})P( u_{k+1})
-\frac{1}{k-1}\prod_{i=1}^{k+1}P(u_{i})\\
&~~-\frac{1}{k-1}\sum_{i=1}^{k}P(P(u^{*}_{i})u_{k+1})
+\frac{1}{k-1}P\bigg(\prod_{i=1}^k P(u_i) u_{k+1}\bigg).
%
\end{align*}
On the right hand side, the second term is canceled with the seventh term and the third with the fifth. Also combine the first term with the last one and move the fourth term moved to the left hand side. Noting that
$$\sum_{i=1}^k P(u_{k,i}^*P(u_{k+1})) + P\left(\prod_{i=1}^kP(u_i) u_{k+1}\right) = \sum_{i=1}^{k+1} P(u_{k+1,i}^*),$$
we obtain
\begin{align*}
\frac{k}{k-1}P\bigg(\prod_{i=1}^{k+1}P(u_{i})\bigg) =\frac{1}{k-1}\sum_{i=1}^{k+1}P(u^{*}_{i})
-\frac{1}{k-1}\prod^{k+1}_{i=1}P(u_{i}).
\end{align*}
This completes the induction.
\end{proof}

\section{Free Reynolds algebras on a set}
\mlabel{sec:free}

In this section we construct the free Reynolds algebra on a set $X$ by using bracketed words in addressing Question~\ref{qu:birk} of R.~Birkhoff. In Section~\mref{ss:pre} we recall some background on bracketed words and identify a subset of bracketed words that will serve as the basis of the free Reynolds algebra to be constructed in Section~\mref{ss:free} where the desired structures and properties of the free Reynolds algebra will be verified.

\subsection{Bracketed words and Reynolds words}
\mlabel{ss:pre}

We first recall from~\mcite{Gop} the construction of the free operated monoid generated by a set.
For any nonempty set $Y$, let $M(Y)$ denote the free monoid generated by $Y$ with the identity $\bfone$. Let $\lfloor Y \rfloor := \{ \lfloor y \rfloor ~|~ y \in Y \}$ be a replica of $Y$. In other words, it is the set $\lfloor Y \rfloor$ indexed by $Y$ but is disjoint from $Y$.

Let $X$ be a nonempty set. We define a direct system as follows. Let
$$
\mathfrak{M}_{0}:= M(X), \quad \mathfrak{M}_{1}:=M(X \sqcup \lfloor \mathfrak{M}_{0} \rfloor ) = M (X \sqcup \lfloor M(X)\rfloor ),
$$
with the natural injection
$$
i_{0,1}: \mathfrak{M}_{0} = M(X) \hookrightarrow \mathfrak{M}_{1} = M(X \sqcup \lfloor \mathfrak{M}_{0} \rfloor).
$$
Inductively assuming that $\mathfrak{M}_{n-1}$ and $
i_{n-2,n-1}: \mathfrak{M}_{n-2} \hookrightarrow \mathfrak{M}_{n-1}$ have been obtained for $n \geq 2$, we define
$$
\mathfrak{M}_{n}:= M (X \sqcup \lfloor \mathfrak{M}_{n-1}\rfloor).
$$
Furthermore, by the freeness of $\mathfrak{M}_{n-1} = M(X \sqcup \lfloor \mathfrak{M}_{n-2} \rfloor)$ as a free monoid, the injection
$$
\lfloor \mathfrak{M}_{n-2} \rfloor  \hookrightarrow \lfloor \mathfrak{M}_{n-1} \rfloor.
$$
induces a monoid homomorphism
$$
\mathfrak{M}_{n-1} = S(X \sqcup \lfloor \mathfrak{M}_{n-2} \rfloor) \hookrightarrow M(X \sqcup \lfloor \mathfrak{M}_{n-1} \rfloor) = \mathfrak{M}_{n}.
$$
Finally, define $$\mathfrak{M}(X):= \dirlim\mathfrak{M}_{n}.$$
Elements in $\mathfrak{M}(X)$ are called {\bf bracketed words} on $X$. Define the {\bf depth} $\dep(w)$ of $w\in\frak M(X)$ to be
$$\dep(w):=\min\{n\mid w\in\frak M_{n}\}.$$
Taking direct limit on both sides of $\frak{M}_n = M(X \sqcup \lfloor \frak{M}_{n-1} \rfloor)$, we obtain
\begin{equation}
\frak{M}(X) = M(X \sqcup \lfloor \mathfrak{M}(X) \rfloor).
\mlabel{eq:Mon}
\end{equation}
Thus every bracketed word $w\neq \bfone$ has a unique decomposition, called the {\bf standard decomposition},
\begin{equation}
w =w_{1}w_{2} \cdots w_{b},
\mlabel{eq:decom}
\end{equation}
where $w_{i}$ is either in $X$ or $\lc\frak M(X)\rc$ for $i = 1, 2, \cdots, b$.

We now identify a subset of bracketed words that will serve as the linear basis of our construction of free Reynolds algebra on a set.

\begin{defn}
Let $X$ be a nonempty set. A bracketed word $w\in \frak M(X)$ is called a {\bf Reynolds word} if $w$ does not contain any subword of the form
$\lc\lc u_{1}\rc\cdots \lc u_{n}\rc\rc=\lc \prod_{i=1}^n \lc u_i\rc \rc$, where $u_{1},\cdots,u_{n}\in \mathfrak{M}(X)$ and $n\geq2$.
\end{defn}

For example, let $X=\{x\}$, then $\lc x\rc\lc x\rc, \lc x\lc x\rc\rc, \lc \lc x \rc\rc, \lc x \rc x\lc x \rc \lc x \rc x,\lc x \rc \lc x \rc \lc x \rc $ are Reynolds words with depths $1, 2, 2, 1, 1$ respectively. While $\lc \lc x \rc \lc x^2 \rc \rc, \lc\lc x \rc \lc x \rc \lc x \rc \rc$ and $\lc \lc\lc x\rc\rc \lc x\rc \rc$ are not Reynolds words.

Similar to Rota-Baxter trees \mcite{ZGG} and average trees \mcite{PG}, we will give a tree representation of the Reynolds words. Further discussions on the combinatorics and enumeration of free Reynolds algebras will be continued in another work.

First, we recall some basic concepts and facts of decorated planar rooted trees. For reference, see \mcite{Foi,Gub,ZGG}.

A {\bf rooted tree} is a connected and simply-connected set of vertices and oriented edges such that there is precisely one distinguish vertex, called the {\bf root}, with no incoming edge. Note that in the following all vertices are represented by a dot and the root vertex is at the bottom of a tree.
A {\bf planar rooted tree} is a rooted tree with a fixed embedding into the plane. Let $\calt$ be the set of planar rooted trees and $M(\calt)$ the free monoid generated by $\calt$ with concatenation product and with the identity $\bfone$. Elements in $\calf:=M(\calt)$ are called {\bf planar rooted forests}. We call the map
$$B^+:\calf\to \calf, \quad F\mapsto B^+(F)=\lc F\rc$$
the {\bf grafting operator} on $\calf$ with $B^+(\bfone)=\lc\bfone\rc=\bullet$. Note that $(\calf,B^+)$ is an operated monoid.

Now, we recall from~\mcite{ZGG} the structure of decorated planar rooted forest. Let $X$ be a set, $\sigmaup$ a symbol not in the set $X$. Denote by $\widetilde{X}:=X\cup\{\sigmaup\}$.
Let $\calt_\ell(\widetilde{X})$ denote the set of vertex decorated trees where elements of $X$ decorate the leaves only. In other words, all internal vertices, as well as possibly some of the leaf vertices, are decorated by $\sigmaup$.
Elements in $T\in\calt_\ell(\widetilde{X})$ (resp. $\calf_\ell(\widetilde{X}):=M(\calt_\ell(\widetilde{X}))$) are called {\bf decorated planar rooted trees} (resp. {\bf decorated planar rooted forests}). We denote the grafting operator
$$B^+_\sigmaup : \calf_\ell(\widetilde{X})\to \calf_\ell(\widetilde{X}), F\mapsto B^+_\sigmaup(F), \tforall F\in \calf_\ell(\widetilde{X}),$$
to be grafting a forest $F$ with a new root decorated by $\sigmaup$, with the convention that $B^+_\sigmaup(\bfone)=\bullet_\sigmaup$.

\begin{prop} \mcite{ZGG}
Let $X$ be a set, $\sigmaup\notin X$ a symbol and $\widetilde{X}:=X\cup\{\sigmaup\}$. Then $(\mathfrak{M}(X),\lc~\rc)$ and $(\calf_\ell(\widetilde{X}),B^+_\sigmaup)$ are free operated monoids on $X$. Hence there is a unique isomorphism of operated monoids
$\Phi: (\mathfrak{M}(X),\lc~\rc)\longrightarrow (\calf_\ell(\widetilde{X}),B^+_\sigmaup)$ sending $x$ to $\bullet_x$.
\end{prop}
Under this isomorphism, Reynolds words can be identified with a class of decorated rooted trees defined as follows.

\begin{defn}
A {\bf Reynolds forest} is a planar decorated rooted forest in $\calf_\ell(\widetilde{X})$ without the following ``crown" shaped subtrees, where $n\geq 2$:
$$
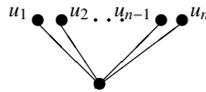
\setlength{\unitlength}{1mm}
\begin{picture}(10,4)(0,0)
\put(-3,0){$^{u_1}\bullet$}
\put(3,0){$\bullet^{u_2}$}
\put(8,0){$\cdots$}
\put(11,0){$^{u_{n-1}}\bullet$}
\put(19,0){$\bullet^{u_n}$}
\put(8,-8.5){$\bullet$}
\put(8.5,-8){$\line(1,1){9}$}
\put(9.5,-8){$\line(-1,1){9}$}
\put(8.5,-8){$\line(4,3){12}$}
\put(9.5,-8){$\line(-2,3){6}$}
\end{picture}
$$
\medskip

\noindent
Here each leaf dot $\bullet_{u_i}$ represents a grafting $B^+_\sigmaup(F)$, hence the subtree is called a {\bf super crown}. A {\bf decorated Reynolds forest} is a decorated planar rooted forest without decorated subtrees of the above form.
\end{defn}

Here are some examples of Reynolds trees decorated by the set $X=\{x\}$.
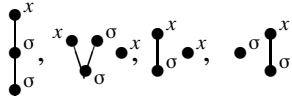
$$\setlength{\unitlength}{1mm}
\begin{picture}(50,10)(0,0)
\put(0,5){$\bullet^{x}$}
\put(0,0){$\bullet^{\sigmaup}$}
\put(0,-5){$\bullet^{\sigmaup}$}
\put(1,5.3){$\line(0,-1){10}$}
\quad\quad
\put(4,0),
\put(6,1.5){$^{x}\bullet$}
\put(11,1.5){$\bullet^{\sigmaup}$}
\put(9.4,-2.5){$\bullet_{\sigmaup}$}
\put(10.2,-2.3){$\line(1,3){1.4}$}
\put(10.3,-2.3){$\line(-1,3){1.4}$}
\put(14.2,0){$\bullet^{x}$}
\quad\qquad\quad
\put(16.5,0),
\put(19,2.5){$\bullet^{x}$}
\put(19,-2.5){$\bullet^{\sigmaup}$}
\put(20,-1){$\line(0,1){5}$}
\put(23,0){$\bullet^{x}$}
\qquad\quad
\put(26,0),
\put(30,0){$\bullet^{\sigmaup}$}
\put(34,2.5){$\bullet^{x}$}
\put(34,-2.5){$\bullet^{\sigmaup}$}
\put(35,-1){$\line(0,1){5}$}
\end{picture}
$$
\medskip

\noindent which represent Reynolds words $\lc\lc\lc x\rc\rc\rc, \lc x\lc \bfone\rc\rc, \lc x\rc x$ and $\lc \bfone\rc\lc x\rc$ respectively. Note that the second forest has a crown but not a super crown since the left branch is not a grafting.
On the other hand, the forests

$$\setlength{\unitlength}{1mm}
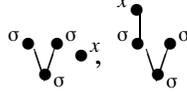
\begin{picture}(40,4)(0,0)
\put(3.5,1.5){$^{\sigmaup}\bullet$}
\put(9,1.5){$\bullet^{\sigmaup}$}
\put(7.4,-2.5){$\bullet_{\sigmaup}$}
\put(8.3,-2.3){$\line(1,3){1.4}$}
\put(8.3,-2){$\line(-1,3){1.3}$}
\put(12.2,0){$\bullet^{x}$}
\put(15,0),
\quad\qquad\quad
\put(18,1.5){$^{\sigmaup}\bullet$}
\put(24,1.5){$\bullet^{\sigmaup}$}
\put(22.4,-2.6){$\bullet_{\sigmaup}$}
\put(23.3,-2.3){$\line(1,3){1.5}$}
\put(23.3,-2.3){$\line(-1,3){1.4}$}
\put(18,6){$^{x}\bullet$}
\put(21,7){$\line(0,-1){5}$}
\end{picture}
$$
\medskip

\noindent
are not decorated Reynolds forests, corresponding to the bracketed words $\lc \lc \bfone\rc \lc \bfone\rc\rc x, \lc \lc x\rc \lc \bfone\rc \rc$.
\smallskip

Now we return to the discussion of Reynolds words. For ease of applications, we next give a recursive structure of the set of Reynolds words. First let
\begin{equation}
\frak R_{0}= M(X),\,\frak R'_{0}= M(X).
\mlabel{eq:Init}
\end{equation}
Suppose that $\frakR_n$ and $\frakR_n'$ have been defined for $n\geq 0$. With the notation
$$ S_{\geq 2} (Y):= \coprod_{m=2}^\infty Y^m$$
for any subset $Y\subseteq \frakM(X)$ and Cartesian power $Y^m$, we recursively define
\begin{equation}
\frak R_{n+1}: = M( X\sqcup \lc \frak R_{n}^{\prime} \rc), \,
\frak R_{n+1}': =\frakR_{n+1} \setminus S_{\geq 2} (\lc \frakR_n'\rc ) =\frakR_{n+1} \setminus S_{\geq 2} (\lc \frakR_n\rc ),
\mlabel{eq:Recu}
\end{equation}
where the last equation follows from the first equation.
The followings are some elementary properties of $\frak R_{n}$ and $\frak R_{n}'$.

\begin{prop}
Let $n\geq0$. Then
\begin{align}
\frak R_{n} \subseteq \frak R_{n+1},\,
\frak R'_{n} \subseteq \frak R'_{n+1}.
\mlabel{eq:IRR'}
\end{align}
\mlabel{prop:IRR'}
\end{prop}
\begin{proof}
We prove the inclusions by induction on $n\geq 0$. When $n = 0$,
by Eq.~(\mref{eq:Recu}) we have
$$\frak R_{0}=M(X)\subseteq M(X\cup\lc\frak R_{0})\rc)=\frak R_{1}\text{ and }\frak R_{0}'=M(X)\subseteq(M(X)\setminus S(\lc M(X)\rc))\sqcup \lc M(X)\rc=\frak R'_{1}.$$
For a fixed $k\geq0$, assume that Eq.~(\mref{eq:IRR'}) holds for $n\leq k$ and consider the case when $n=k+1$. By the induction hypothesis,
$$\frak R_{k} \subseteq \frak R_{k+1},\,\frak R'_{k} \subseteq \frak R'_{k+1}.$$
Thus
\begin{equation}
\lc\frak R'_{k}\rc\subseteq\lc\frak R'_{k+1}\rc.
\mlabel{eq:BR'}
\end{equation}
So we have
\begin{equation}
\frak R_{k+1} = M( X\sqcup \lc \frak R'_{k} \rc)\subseteq M( X\sqcup \lc \frak R'_{k+1}\rc)=\frak R_{k+2},
\mlabel{eq:Rr}
\end{equation}
completing the inductive proof of the first inclusion in Eq.~\eqref{eq:IRR'}.

Next from $\frakR_{k}=M(X\sqcup \lc \frakR_{k-1}'\rc)$, a bracket $\lc u\rc$ can appear in $\frakR_{k}$ only when $u\in \frakR_{k-1}'$. Thus \begin{equation}
\frak R_{k}\cap S(\lc\frak R_{k}'\rc)=S (\lc\frak R'_{k-1}\rc)
\mlabel{eq:Cap}
\end{equation}
and hence
\begin{equation}
\frak R_{k}\cap S_{\geq 2}(\lc\frak R_{k}'\rc)=S_{\geq 2} (\lc\frak R_{k-1}'\rc).
\mlabel{eq:RR'}
\end{equation}
Therefore,
\begin{align*}
\frak R'_{k+1}
&= \frak R_{k+1}\setminus S_{\geq 2} (\lc\frak R'_{k}\rc)\quad(\text{by Eq.~(\mref{eq:Recu})})\\
&\supseteq \frakR_k \setminus S_{\geq 2} (\lc \frakR_k'\rc)
\quad (\text{by the first inclusion in Eq.~\eqref{eq:IRR'}})\\
&= \frak R_{k}\setminus S_{\geq 2} (\lc\frak R_{k-1}'\rc) \quad(\text{by Eq.~(\mref{eq:RR'})})\\
&= \frakR_k'.
\end{align*}
This finishes the proof of the second inclusion in Eq.~\eqref{eq:IRR'}.
\end{proof}

From Proposition \mref{prop:IRR'}, there are two direct systems
$$\{\frak R_{n},\iota_{n,n+1}:\frak R_{n}\hookrightarrow\frak R_{n+1}\}^{\infty}_{n=0}\text{ and~ }\{\frak R'_{n},\iota_{n,n+1}:\frak R'_{n}\hookrightarrow\frak R'_{n+1}\}^{\infty}_{n=0},$$
where $\iota_{n,n+1}$ is the natural inclusion. Furthermore, for the direct limits, we have
\begin{equation}
\frak R:= \dirlim\frak R_{n}=\bigcup_{n\geq0}\frak R_{n},\,
\frak R':= \dirlim \frak R'_{n}=\bigcup_{n\geq0}\frak R'_{n}.
\mlabel{eq:Dirlim}
\end{equation}
By taking direct limit on both sides of Eq.~\eqref{eq:Recu}, we obtain
\begin{equation}
\frak R = M( X\sqcup \lfloor \frak R' \rfloor),\ \frak R'=\frak R\setminus S_{\geq 2} (\lc\frak R\rc).
\mlabel{eq:DirR}
\end{equation}

\subsection{Construction of the free Reynolds algebra on a set $X$}
\mlabel{ss:free}

We will construct the free Reynolds algebra on a set $X$ by equipping the free module $\bfk \frak R$ spanned by the set $\frak R=\frak R(X)$ of Reynolds words with a suitably defined multiplication and a linear operator. The multiplication is easy to define. But the linear operator with the desired properties is not straightforward and will take up the rest of the discussion. We will assume that $\bfk$ is a field with characteristic zero in the following.

Applying the free monoid structure on $X\sqcup \lc \frakR'\rc$ in~Eq.(\mref{eq:DirR}) where the multiplication is the concatenation, we equip the free module $\bfk\frak R$ with the free noncommutative polynomial algebra
$$\bfk\frakR=\bfk\langle X\sqcup \lc \frakR'\rc\rangle.$$

Next we define a linear operator $P:\bfk\frak R\rightarrow \bfk\frakR$ and show that it is a Reynolds operator. Denote $$\frakR'':=\frakR\setminus \frakR',$$
so that $\frakR=\frakR'\sqcup \frakR''$ and $\frakR''\subseteq S_{\geq 2}(\lc \frakR\rc)$. For $r$ in $\frak R$, we split the definition of $P(r)$ into two cases.
\smallskip

\noindent{\bf Case 1}. If $r\in\frak R'$, then $\lc r\rc$ is in $\frakR$ and we define
\begin{equation}
P(r):=\lc r\rc.
\mlabel{eq:R'}
\end{equation}

\noindent{\bf Case 2}. If $r\in\frak R''$, then we have $r=\lc s_1\rc \lc s_2\rc\cdots \lc s_m\rc, s_i\in \frakR, 1\leq i\leq m, m\geq 2$. In fact, $s_i$ are in $\frakR'$. By extracting the largest number of brackets from $s_i, 1\leq i\leq m$, we can uniquely write
$$r =\lc r_{1}\rc^{(n_{1})}\cdots \lc r_m\rc^{(n_m)},$$
where $r_{j}\in\frak R'\setminus \lc \frak R\rc, n_{j}\geq1$ for $1\leq j\leq m$. With this notation, we define $P(r)$ by induction on $n:=n_{1}+\cdots+n_{m}$. Then $n\geq m\geq2$. When $n=m$, then $n_{1}=\cdots=n_{m}=1$. So $r =\prod_{j=1}^{m}\lc r_{j}\rc$ with $r_{1},\cdots, r_{m}\in\frak R'\setminus\lc \frak R'\rc\subset\frak R'$. Then for all $1\leq i \leq n$, we have
$$\lc r_{1}\rc\cdots\lc r_{i-1}\rc r_{i}\lc r_{i+1}\rc\cdots\lc r_{m}\rc\in\frak R'.$$
Then we define
\begin{equation}
P(r):=\frac{1}{m-1}\left(\sum^{m}_{i=1}\lc\lc r_{1}\rc\cdots\lc r_{i-1}\rc r_{i}\lc r_{i+1}\rc\cdots\lc r_{m}\rc\rc-r\right).
\mlabel{eq:R''0}
\end{equation}

For $k\geq m\geq2$, assume that $P(r)$ have been defined when $m\leq n\leq k$ and consider $P(r)$ with $n=n_{1}+\cdots+n_{m}=k+1$. Then define
\begin{equation}
P(r):=\frac{1}{m-1}\left(\sum^{m}_{i=1}P(\lc r_{1}\rc^{(n_{1})}\cdots\lc r_{i-1}\rc^{(n_{i-1})} \lc r_{i}\rc^{(n_{i}-1)}\lc r_{i+1}\rc^{(n_{i+1})}\cdots\lc r_{m}\rc^{(n_{m})})-r\right).
\mlabel{eq:R''}
\end{equation}
Since $$n_{1}+\cdots+n_{i-1}+(n_{i}-1)+n_{i+1}+\cdots+n_{m}=k,$$
$P(\lc r_{1}\rc^{(n_{1})}\cdots\lc r_{i-1}\rc^{(n_{i-1})} \lc r_{i}\rc^{(n_{i}-1)}\lc r_{i+1}\rc^{(n_{i+1})}\cdots\lc r_{m}\rc^{(n_{m})})$ is well defined either by Case 1 (when $n_i-1=0$) or by the induction hypothesis. Hence $P(r)$ is well-defined.
We then extend $P$ to $\bfk\frak R $ by linearity.

In summary, we have following algorithm for the definition of $P$:

\begin{algo}
{\bf (Algorithm of $P$)} Let $r$ be in $\frakR$.
\begin{enumerate}
\item
If $r$ is in $\frakR'$, then define $P(r):=\lc r\rc$ and stop;
\item
If $r$ is not in $\frakR'$, then apply Eq.~\eqref{eq:R''0} to $r$ as many times as needed, until the resulting expression $\tilde{r}$ is in $\bfk \frakR'$. Then invoke Eq.~\eqref{eq:R'}, namely with all $P(u)$ replaced by $\lc u \rc$.
\end{enumerate}
\end{algo}

As an illustration of this procedure, we give the following example. Comparing with Example~\mref{ex:R}, we see that we have defined $P$ by essentially axiomatizing Proposition~\mref{prop:R}.

\begin{exam}
Let $x, y,z\in X$ and $r=\lc x\rc^{(2)}\lc y\rc^{(1)}\lc z\rc^{(1)}=\lc\lc x\rc \rc \lc y\rc \lc z\rc\in\frak R''$. Then
\begin{align*}
P(r)&=\frac{1}{2}P( \lc x\rc \lc y\rc\lc z\rc )+\frac{1}{2}P(\lc\lc x\rc\rc y\lc z\rc)+\frac{1}{2}P(\lc\lc x\rc\rc\lc y\rc z)-\frac{1}{2}r\quad(\text{by Eq.~(\mref{eq:R''})})\\
&=\frac{1}{4}P(  x \lc y\rc\lc z\rc )+\frac{1}{4}P( \lc x\rc  y\lc z\rc )+\frac{1}{4}P( \lc x\rc \lc y\rc z )-\frac{1}{4} \lc x\rc \lc y\rc\lc z\rc \\
&\ \ \ +\frac{1}{2}P(\lc\lc x\rc\rc y\lc z\rc)+\frac{1}{2}P(\lc\lc x\rc\rc\lc y\rc z)-\frac{1}{2}r\quad(\text{by Eq.~(\mref{eq:R''}) applied to the first term})\\
&=\frac{1}{4}\lc x \lc y\rc\lc z\rc \rc+\frac{1}{4}\lc \lc x\rc  y\lc z\rc \rc+\frac{1}{4}\lc \lc x\rc \lc y\rc z \rc-\frac{1}{4} \lc x\rc \lc y\rc\lc z\rc\\
&\ \ \ +\frac{1}{2}\lc\lc\lc x\rc\rc y\lc z\rc\rc+\frac{1}{2}\lc\lc\lc x\rc\rc\lc y\rc z\rc-\frac{1}{2}r\quad(\text{by Eq.~(\mref{eq:R'})}).
\end{align*}
\end{exam}

\begin{prop}
Let $\bfk$ be a field with characteristic zero. With the concatenation multiplication and with the operator $P$ defined by Eqs.~$($\mref{eq:R''}$)$ and $($\mref{eq:R'}$)$, the pair $(\bfk\frak R, P)$ is a Reynolds algebra.
\mlabel{prop:RA}
\end{prop}

\begin{proof}

Since the concatenation multiplication is associative, we only need to verify the Reynolds identity:
\begin{equation}
P\bigg(P(r)P(s)\bigg)=P\bigg(P(r)s\bigg)+P\bigg(rP(s)\bigg)-P(r)P(s) \tforall r, s\in \bfk\frakR.
\mlabel{eq:ROP1}
\end{equation}
We first introduce a preparational lemma.

Let $r =\lc r_{1}\rc^{(n_{1})} \cdots \lc r_{m_{r}}\rc^{(n_{m_{r}})}, s=\lc s_{1}\rc^{(l_{1})} \cdots \lc s_{m_{s}}\rc^{(l_{m_{s}})}$ be in $\frak R''$ with $r_{i}, s_{j}\in\frak R'\setminus\lc\frak R'\rc, n_{i}\geq 1, l_{j}\geq 1$, where $1\leq i\leq m_{r},1\leq j\leq m_{s}$. To simplify notations, we will use the abbreviations
\begin{align*}
r^{\ast}_{i}:=r^\ast_{m_r,i}&:=\lc r_{1}\rc^{(n_{1})}\cdots\lc r_{i-1}\rc^{(n_{i-1})} \lc r_{i}\rc^{(n_{i}-1)}\lc r_{i+1}\rc^{(n_{i+1})}\cdots\lc r_{m_{r}}\rc^{(n_{m_{r}})},\\
s^{\ast}_{j}:=s^\ast_{m_s,j}&:=\lc s_{1}\rc^{(l_{1})}\cdots\lc s_{j-1}\rc^{(l_{j-1})} \lc s_{j}\rc^{(l_{j}-1)}\lc s_{j+1}\rc^{(l_{j+1})}\cdots\lc s_{m_{s}}\rc^{(l_{m_{s}})}.
\end{align*}

\begin{lemma}
Let $r, s\in \frakR=\frakR' \sqcup \frakR''$ be Reynolds words. Suppose that $r$ is in $\frakR''$, so $r =\lc r_{1}\rc^{(n_{1})} \cdots \lc r_{m_{r}}\rc^{(n_{m_{r}})}\in\frak R''$ with $r_{i}\in\frak R'\setminus\lc\frak R'\rc, n_{i}\geq1$, where $1\leq i\leq m_{r}$ and $m_r\geq 2$. Then
\begin{equation}
m_{r}P\bigg(r P(s)\bigg)
=\sum_{i=1}^{m_{r}}P\bigg(r^{\ast}_{i}P(s)\bigg)-r P(s)+P(r s).
\mlabel{eq:SE'}
\end{equation}
Similarly, suppose that $s$ is in $\frakR''$, so $s=\lc s_{1}\rc^{(l_{1})} \cdots \lc s_{m_{s}}\rc^{(l_{m_{s}})}\in\frak R''$ with $s_{j}\in\frak R'\setminus\lc\frak R'\rc, l_{j}\geq1$, where $1\leq j\leq m_{s}$ and $m_s\geq 2$. Then
\begin{equation}
m_{s}P\bigg(P(r)s\bigg)
=\sum_{j=1}^{m_{s}}P\bigg(P(r)s^{\ast}_{j}\bigg)-P(r)s+P(r s).
\label{eq:SE}
\end{equation}
\mlabel{lem:C}
\end{lemma}
We now continue with the proof of Proposition~\mref{prop:RA} assuming Lemma~\mref{lem:C} whose proof is postponed to Section~\mref{ss:lemma}.

Since $P$ is linear, we just need to verify that Eq.~(\mref{eq:ROP1}) holds for $r,s\in \frak R=\frak R'\sqcup\frak R''$. Let $m:=\deg_{P}(r)$ and $n:=\deg_{P}(s)$. Then $m\geq0$ and $n\geq0$. We verify Eq.~(\mref{eq:ROP1}) by induction on $p:=m+n$. When $p=0$, we have $m=n=0$, that is, $r,s\in\frak R_{0}=M(X)$. By Eq.~(\mref{eq:R'}), we obtain
$$P(r)=\lc r\rc\text{ and }P(s)=\lc s\rc.$$
Thus
\begin{align*}
P\bigg(P(r)P(s)\bigg)
&=P\bigg(\lc r\rc\lc s\rc\bigg)\\
&=P(\lc r\rc s)+P( r\lc s\rc)-\lc r\rc\lc s\rc\quad(\text{by Eq. }(\mref{eq:R''}))\\
&=\lc\lc r\rc s\rc+\lc r\lc s\rc\rc-\lc r\rc\lc s\rc\\
&=P\bigg(P(r)s\bigg)+P\bigg(rP(s)\bigg)-P(r)P(s).\quad(\text{by Eq. }(\mref{eq:R'}))
\end{align*}

For any given $k\geq0$, assume that Eq.~(\mref{eq:ROP1}) holds for $m+n=p\leq k$, and consider the case when $p=k+1$. There are four cases to consider depending on $r$ and/or $s$ are in $\frakR'$ or $\frakR''$.

\noindent{\bf Case 1.} $r,s\in\frak R'$: Applying Eq.~(\mref{eq:R'}), we have $P(r)=\lc r\rc, P(s)=\lc s\rc$. Thus Eq.~(\mref{eq:ROP1}) is obtained by the same argument as above, applying Eqs. (\mref{eq:R'}) and (\mref{eq:R''}).
\smallskip

\noindent{\bf Case 2.} $r =\lc r_{1}\rc^{(n_{1})} \cdots \lc r_{m_{r}}\rc^{(n_{m_{r}})}\in\frak R'', s\in\frak R'$ where $r_{i}\in\frak R'\setminus\lc\frak R'\rc, n_{i}\geq2$ for $1\leq i\leq m_{r}, m_r\geq 2$: By Eqs. (\mref{eq:R'}) and (\mref{eq:R''}), we obtain
$$
P(r)=\frac{1}{m_{r}-1}\left(\sum^{m_r}_{i=1}P(r^{\ast}_{i})-r\right),\quad P(s)=\lc s\rc.
$$
From these equations, we obtain
\allowdisplaybreaks{
\begin{align*}
P\bigg(P(r)P(s)\bigg)
&=\frac{1}{m_{r}-1}\sum_{i=1}^{m_{r}}P\bigg(P(r^{\ast}_{i})P(s)\bigg)
-P(r P(s))\bigg)
\quad(\text{by linearity of } P)\\
&=\frac{1}{m_{r}-1}\left(\sum_{i=1}^{m_{r}}P\bigg(P(r^{\ast}_{i})P(s)\bigg)
-P(r\lc s\rc)\right)\\
&=\frac{1}{m_{r}-1}\sum_{i=1}^{m_{r}}P(P(r^{\ast}_{i})s)
+\frac{1}{m_{r}-1}\sum_{i=1}^{m_{r}}P(r^{\ast}_{i}\lc s\rc)
-\frac{1}{m_{r}-1}\sum_{i=1}^{m_{r}}P(r^{\ast}_{i})\lc s\rc
-\frac{1}{m_{r}-1}P(r\lc s\rc)\\
&\hspace{2cm}(\text{by the induction hypothesis })
\end{align*}
}
and
\allowdisplaybreaks{
\begin{align*}
&P\bigg(P(r)s\bigg)+P\bigg(rP(s)\bigg)-P(r)P(s)\\
&=\frac{1}{m_{r}-1}\sum_{i=1}^{m_{r}}P\bigg(P(r^{\ast}_{i}) s\bigg)
-\frac{1}{m_{r}-1}P(r s)+P\bigg(r\lc s\rc\bigg)\\
&~~~~~~~~-\frac{1}{m_{r}-1}\sum_{i=1}^{m_{r}}P(r^{\ast}_{i})\lc s\rc
+\frac{1}{m_{r}-1}r \lc s\rc
\quad(\text{by Eq. }(\mref{eq:R''})).
\end{align*}}

Taking the difference of the two equations and simplifying, we obtain

\begin{align*}
& P\bigg(P(r)P(s)\bigg) - P\bigg(P(r)s\bigg)-P\bigg(rP(s)\bigg)+P(r)P(s) \\
&=\frac{1}{m_{r}-1}\sum_{i=1}^{m_{r}}P(r^{\ast}_{i}\lc s\rc)
-\frac{m_{r}}{m_{r}-1}P(r\lc s\rc)-\frac{1}{m_{r}-1}r \lc s\rc +\frac{1}{m_{r}-1}P(r s),
\end{align*}
which vanishes by Lemma~\mref{lem:C}.
So Eq.~(\mref{eq:ROP1}) holds in this case.
\smallskip

\noindent{\bf Case 3.} $r\in\frak R', s=\lc s_{1}\rc^{(l_{1})} \cdots \lc s_{m_{s}}\rc^{(l_{m_{s}})}\in\frak R''$ where $s_{j}\in\frak R'\setminus\lc\frak R'\rc, l_{j}\geq2$ for $1\leq j\leq m_{s}$: It follows from a proof similar to Case 2, by applying Lemma~\mref{lem:C}.
\smallskip

\noindent{\bf Case 4.} $r =\lc r_{1}\rc^{(n_{1})} \cdots \lc r_{m_{r}}\rc^{(n_{m_{r}})}, s=\lc s_{1}\rc^{(l_{1})} \cdots \lc s_{m_{s}}\rc^{(l_{m_{s}})}\in\frak R''$ where $r_{i}, s_{j}\in\frak R'\setminus\lc\frak R'\rc, n_{i}\geq2, l_{j}\geq2$ for $1\leq i\leq m_{r},1\leq j\leq m_{s}$: By Eq.~(\mref{eq:R''}), we obtain
$$
P(r)=\frac{1}{m_{r}-1}\left(\sum_{i=1}^{m_{r}}P(r^*_{i})-r\right), \quad
P(s)=\frac{1}{m_{s}-1}\left(\sum_{j=1}^{m_{s}}P(s^*_{j})-s\right).
$$
Consequently,
\begin{align*}
P\bigg(P(r)P(s)\bigg)
&=\frac{1}{m_{r}-1}\frac{1}{m_{s}-1}\sum_{i=1}^{m_{r}}\sum_{j=1}^{m_{s}}P\bigg(
P(r^{\ast}_{i})P(s^{\ast}_{j})\bigg)+\frac{1}{m_{r}-1}\frac{1}{m_{s}-1}P(rs)\\
&~~~~~~~~-\frac{1}{m_{r}-1}\frac{1}{m_{s}-1}\sum_{i=1}^{m_{r}}P\bigg(P(r^{\ast}_{i})s\bigg)-\frac{1}{m_{r}-1}
\frac{1}{m_{s}-1}\sum_{j=1}^{m_{s}}P\bigg(r P(s^{\ast}_{j})\bigg)\\
&\hspace{2cm}(\text{by the linearity of } P)\\
&=\frac{1}{m_{r}-1}\frac{1}{m_{s}-1}\sum_{i=1}^{m_{r}}\sum_{j=1}^{m_{s}}P\bigg(r^{\ast}_{i}P(s^{\ast}_{j})\bigg)
+\frac{1}{m_{r}-1}\frac{1}{m_{s}-1}\sum_{i=1}^{m_{r}}\sum_{j=1}^{m_{s}}P\bigg(P(r^{\ast}_{i})s^{\ast}_{j}\bigg)\\
&~~~~~~~~-\frac{1}{m_{r}-1}\frac{1}{m_{s}-1}\sum_{i=1}^{m_{r}}\sum_{j=1}^{m_{s}}P(r^{\ast}_{i})
P(s^{\ast}_{j})
-\frac{1}{m_{r}-1}\frac{1}{m_{s}-1}\sum_{i=1}^{m_{r}}P\bigg(P(r^{\ast}_{i})s\bigg)\\
&~~~~~~~-\frac{1}{m_{r}-1}\frac{1}{m_{s}-1}\sum_{j=1}^{m_{s}}P\bigg(r P(s^{\ast}_{j})\bigg)
+\frac{1}{m_{r}-1}\frac{1}{m_{s}-1}P(rs)\\
&\hspace{2cm}(\text{by the induction hypothesis})
\end{align*}
and
\begin{align*}
&P\bigg(P(r)s\bigg)+P\bigg(rP(s)\bigg)-P(r)P(s)\\
&=\frac{1}{m_{r}-1}\sum_{i=1}^{m_{r}}P\bigg(P(r^{\ast}_{i}) s\bigg)
+\frac{1}{m_{s}-1}\sum_{j=1}^{m_{s}}P\bigg(r P(s^{\ast}_{j})\bigg)
-\frac{m_{r}+m_{s}-2}{(m_{s}-1)(m_{r}-1)}P(r s)\\
&~~~~~~~~-\frac{1}{m_{r}-1}\frac{1}{m_{s}-1}\sum_{i=1}^{m_{r}}\sum_{j=1}^{m_{s}}P(r^{\ast}_{i})P(s^{\ast}_{j})
-\frac{1}{m_{r}-1}\frac{1}{m_{s}-1}rs\\
&~~~~~~~~+\frac{1}{m_{r}-1}\frac{1}{m_{s}-1}\sum_{i=1}^{m_{r}}P(r^{\ast}_{i})s
+\frac{1}{m_{r}-1}\frac{1}{m_{s}-1}\sum_{j=1}^{m_{s}}r P(s^{\ast}_{j}).
\end{align*}

Thus in order for the Reynolds identity in Eq.~\eqref{eq:ROP1} to hold, we only need to show that the right hand sides of the above two equations are the same. Equating them, clearing the denominators and combining similar terms, we see that we only need to verify

\begin{eqnarray*}
&&\sum_{i=1}^{m_{r}}\sum_{j=1}^{m_{s}}P\bigg(r^{\ast}_{i}P(s^{\ast}_{j})\bigg)
+\sum_{i=1}^{m_{r}}\sum_{j=1}^{m_{s}}P\bigg(P(r^{\ast}_{i})s^{\ast}_{j}\bigg)
-\sum_{i=1}^{m_{r}}P\bigg(P(r^{\ast}_{i})s\bigg)
-\sum_{j=1}^{m_{s}}P\bigg(r P(s^{\ast}_{j})\bigg)
+(m_{r}+m_{s}-1)P(r s)\\
&=&m_s\sum_{i=1}^{m_{r}}P\bigg(P(r^{\ast}_{i}) s\bigg)
+m_r\sum_{j=1}^{m_{s}}P\bigg(r P(s^{\ast}_{j})\bigg)
+\sum_{j=1}^{m_{s}}r P(s^{\ast}_{j})
+\sum_{i=1}^{m_{r}}P(r^{\ast}_{i})s
-rs.
\end{eqnarray*}
Applying Eq.~(\mref{eq:R''}) to $(m_r+m_s-1)P(rs)$ and applying Eqs.~(\mref{eq:SE}) and (\mref{eq:SE'}) in Lemma~\mref{lem:C} to $m_s\sum_{i=1}^{m_{r}}P\big(P(r^{\ast}_{i}) s\big)$
and $m_r\sum_{j=1}^{m_{s}}P\big(r P(s^{\ast}_{j})\big)$ respectively, we see that the equation indeed holds. This completes the inductive proof of Proposition~\mref{prop:RA}.
\end{proof}

We next verify the universal property of the Reynolds algebra $(\bfk\frakR,P)$ in Proposition~\mref{prop:RA}. As noted in the Introduction, this provides a solution to Question~\ref{qu:birk} posted by G. Birkhoff~\cite{Bi}.

\begin{theorem}
Let $\bfk$ be a field with characteristic zero and let $X$ be a set. The Reynolds algebra $(\bfk\mathfrak R, P)$ in Proposition~\mref{prop:RA}, together with the map $i: X\hookrightarrow\bfk\frak R$ defined in Eq.~(\mref{eq:IF}), is the free Reynolds algebra on the set $X$, characterized by the universal property: for any Reynolds algebra $A$ and any set map $f:X\rightarrow A$, there exists a unique Reynolds algebra homomorphism $\bar{f}: \bfk\frak R\rightarrow A$ such that $\bar{f}\circ i=f$, that is, the following diagram commutes
$$
\xymatrix{
  X \ar[dr]_{f} \ar@{^{(}->}[r]^{i}
  & \bfk\frak R \ar@{.>}[d]^{\bar{f}}  \\
  & A }
$$\mlabel{thm:main}
\end{theorem}
\begin{proof}
Let $(A,Q)$ be a Reynolds algebra and let $f:X\to A$ be a set map. We will construct an algebra homomorphism $\bar{f}:\bfk\frakR\to A$ as the direct limit of a direct system $\bar{f}_n:\bfk \frakR_n\to A, n\geq 0$.

To begin with, by the universal property of $\bfk\frak R_{0} = \bfk M(X)$ as the free algebra on $X$, we obtain an algebra homomorphism $\bar{f}_{0}: \bfk\frak R_{0}\rightarrow A$ such that the following diagram commutes
$$
\xymatrix{
  X \ar[dr]_{f} \ar@{^{(}->}[r]^{i_{0}}
  & \bfk\mathfrak R_{0} \ar@{.>}[d]^{\bar{f}_{0}}  \\
  & A }
$$

Then define a map $f_1$ by
$$
 f_{1} : X \sqcup\lc\frak R'_{0}\rc \longrightarrow A,
\quad \left\{\begin{array}{ll} x\mapsto f(x),
&x\in X, \\
\lc x'_{0}\rc\mapsto Q(\bar{f}_{0}(x'_{0})),
&x'_{0}\in \frak R'_{0}.
\end{array}\right.
$$
By the universal property of $\bfk\frak R_{1} =\bfk M(X\sqcup\lc\frak R'_{0}\rc)$ as the free algebra on $X\sqcup\lc\frak R'_{0}\rc$, we obtain an algebra homomorphism $\bar{f}_{1}: \bf k\mathfrak R_{0}\rightarrow A$ making the following diagram commutative
$$
\xymatrix{
  X\sqcup \lc R_{0}^{\prime}\rc \ar[dr]_{f_{1}} \ar@{^{(}->}[r]^{i_{1}}
  & \bf k\mathfrak R_{1} \ar@{.>}[d]^{\bar{f}_{1}}  \\
  & A }
$$

For $k\geq 0$, suppose that algebra homomorphisms $\bar{f}_{n}: \bfk\mathfrak R_{n}\rightarrow A$ for $0\leq n\leq k$ have been defined. We define a map $f_{k+1}$ by
\begin{equation}
 f_{k+1} : X \sqcup\lc\frak R'_{k}\rc\longrightarrow A,
\quad \left\{\begin{array}{ll} x\mapsto f(x),
&x\in X, \\
\lc x'_{k}\rc\mapsto Q(\bar{f}_{k}( x'_{k})),
& x_{k}\in \frak R'_{k}.
\end{array}\right.
\mlabel{eq:OM}
\end{equation}
Then by the universal property of $\bfk\frak R_{k+1}=\bfk M(X\sqcup \lc\frak R_{k}'\rc)$ as the free algebra on $X\sqcup\lc\frak R_{k}'\rc$, we obtain an algebra homomorphism $\bar{f}_{k+1}: \bf k\mathfrak R_{k+1}\rightarrow A$ making the following diagram
commutative
$$
\xymatrix{
  X\sqcup \lc R_{k}^{\prime}\rc \ar[dr]_{f_{k+1}} \ar@{^{(}->}[r]^{i_{k+1}}
  & \bf k\mathfrak R_{k+1} \ar@{.>}[d]^{\bar{f}_{k+1}}  \\
  & A }
$$

We next check the compatibility of $\{\bar{f}_{n}\}_{n\geq0}$ by a recursion on $n$, that is, the following diagram commutes
\begin{align*}
\xymatrix@R=0.5cm{
{\bf k }\mathfrak {R}_{n} \ar[dd]_{\bar{f}_{n}} \ar@{^{(}->}[dr]^{i_{n+1}}\\
                & {\bf k} \mathfrak {{R}}_{n+1} \ar[dl]^{\bar{f}_{n+1}}\\
  A }
\end{align*}
When $n=0$, for any $r\in\frak R_{0}$, we obtain $\bar{f}_{0}(r)=f(r)$ and we also have $\bar{f}_{1}(r)=f(r)=\bar{f}_{1}(i(r))$, as needed.

Assume that the compatibility holds for $0\leq n\leq k$ with $k\geq0$ and consider the case when $n=k+1$. By the induction hypothesis, we have
$$\bar{f}_{k+1}i_{k+1}=\bar{f}_{k}.$$
For $r\in\frak R_{k+1}=M(X\sqcup\lc\frak R'_{k}\rc)$, we need to consider the following two cases:

\noindent{\bf Case 1}. $r\in\frak R'_{k+1}$: Then $\bar{f}_{k+2}i_{k+2}(r)=\bar{f}_{k+2}(r)= f(r)=\bar{f}_{k+1}(r)$.

\noindent{\bf Case 2}. $r =\lc r_{1}\rc^{(n_{1})} \cdots \lc r_{m_{r}}\rc^{(n_{m_{r}})}\in\frak R''_{k+1}$ with $r_{i}\in\frak R'\setminus\lc\frak R'\rc, n_{i}\geq1$, where $1\leq i\leq m_{r}$:
Then we have
\begin{align*}
\bar{f}_{k+2}i_{k+2}(r)
&=\bar{f}_{k+2}i_{k+2}(\lc r_{1}\rc^{(n_{1})} \cdots \lc r_{m_{r}}\rc^{(n_{m_{r}})})\\
&=\bar{f}_{k+2}(\lc r_{1}\rc^{(n_{1})} \cdots \lc r_{m_{r}}\rc^{(n_{m_{r}})})\\
&=\bar{f}_{k+2}(\lc r_{1}\rc^{(n_{1})}) \cdots \bar{f}_{k+2}(\lc r_{m_{r}}\rc^{(n_{m_{r}})})\quad(\text{$\bar{f}_{k+2}$ is a morphism of algebras })\\
&=Q(\bar{f}_{k+1}(\lc r_{1}\rc^{(n_{1}-1)})\cdots Q(\bar{f}_{k+1}(\lc r_{m_{r}}\rc^{(n_{m_{r}}-1)}))
\quad(\text{by Eq.}(\mref{eq:OM}))\\
&=Q(\bar{f}_{k+1}i_{k+1}(\lc r_{1}\rc^{(n_{1}-1)}))\cdots Q(\bar{f}_{k+1}i_{k+1}(\lc r_{m_{r}}\rc^{(n_{m_{r}}-1)}))\\
&=Q(\bar{f}_{k}(\lc r_{1}\rc^{(n_{1}-1)}))\cdots Q(\bar{f}_{k}(\lc r_{m_{r}}\rc^{(n_{m_{r}}-1)}))
\quad(\text{by the induction hypothesis})\\
&=\bar{f}_{k+1}(r)\quad(\text{by Eq. }(\mref{eq:OM})).
\end{align*}

This completes the inductive proof of the compatibility of the algebra homomorphisms $f_n$ on the direct system $\bfk\frakR_n, n\geq 0$. Therefore, we can take direct limit and obtain
\begin{equation}
\bar{f}:=\dirlim\bar{f}_{n}:\bfk\frak R\rightarrow A,
\mlabel{eq:IF}
\end{equation}
which satisfies the following commutative diagram
$$
\xymatrix{
  X \ar[dr]_{f} \ar@{^{(}->}[r]^{i}
  & \bf k \mathfrak R \ar@{.>}[d]^{\bar{f}} \\
  & A}
$$

We further check that $\bar{f}:\bfk\frak R\rightarrow A$ is a Reynolds algebra homomorphism. From Eq.~(\mref{eq:ROP1}), we just need to verify the equation
\begin{equation}
\bar{f}(P(r))=Q(\bar{f}(r))\ \tforall r\in\frak R.
\mlabel{eq:Rmor}
\end{equation}
For this we consider the following two cases:

\noindent{\bf Case 1}. $r\in\frak R'$: Then applying Eqs.~(\mref{eq:R'}) and (\mref{eq:OM}), we obtain
$$\bar{f}(P(r))
=\bar{f}(\lc r\rc) =Q(\bar{f}(r)).
$$

\noindent{\bf Case 2}. $r =\lc r_{1}\rc^{(n_{1})}\cdots \lc r_{m_{r}}\rc^{(n_{m_{r}})}\in\frak R''$: Then $m_r\geq2,\ n_{1},\cdots, n_{m_{r}}\geq1$ and $r_1,\cdots, r_{m_{r}}\in\frak R'\setminus\lc\frak R'\rc$. We now verify Eq.~(\mref{eq:Rmor}) by induction on $n:=n_{1}+\cdots+n_{m_{r}}$. Then $n\geq m_r\geq2$. When $n=2$, we have $m_r=2,n_{1}=1,n_{2}=1$ and then $r =\lc r_{1}\rc\lc r_{2}\rc$. Thus
\begin{align*}
\bar{f}(P(r))
&=\bar{f}(P(\lc r_{1}\rc\lc r_{2}\rc))\\
&=\bar{f}\left(P(r_{1}\lc r_{2}\rc)+P(\lc r_{1}\rc r_{2})-r\right)
\quad(\text{by Eq. }(\mref{eq:R''}))\\
&=\bar{f}\left(\lc r_{1}\lc r_{2}\rc\rc +\lc \lc r_{1}\rc r_{2}\rc -r\right)
\quad(\text{by Eq. }(\mref{eq:R'}))\\
&=\bar{f}(\lc r_{1}\lc r_{2}\rc\rc) +\bar{f}(\lc \lc r_{1}\rc r_{2}\rc) -\bar{f}(\lc r_{1}\rc)\bar{f}(\lc r_{2}\rc)
\quad(\text{$\bar{f}$ is an algebra homomorphism})\\
&=Q(\bar{f}(r_{1}\lc r_{2}\rc)) +Q(\bar{f}(\lc r_{1}\rc r_{2})) -Q(\bar{f}(r_{1}))Q(\bar{f}(r_{2}))
\quad(\text{by Eq. }(\mref{eq:OM}))\\
&=Q(\bar{f}(r_{1})Q(\bar{f}(r_{2}))
+Q(Q(\bar{f}(r_{1})) \bar{f}(r_{2}))
-Q(\bar{f}(r_{1}))Q(\bar{f}(r_{2}))\\
&=Q(Q(\bar{f}(r_{1}))Q(\bar{f}(r_{2})))
\quad(\text{$Q$ is a Reynolds operator})\\
&=Q(\bar{f}(\lc r_{1}\rc)\bar{f}(\lc r_{2}\rc))
\quad(\text{by Eq. }(\mref{eq:OM}))\\
&=Q(\bar{f}(\lc r_{1}\rc\lc r_{2}\rc))
\quad(\text{$\bar{f}$ is an algebra homomorphism})\\
&=Q(\bar{f}(r)).
\end{align*}

Assume that Eq.~(\mref{eq:Rmor}) holds for $r\leq k$ where $k\geq2$ and consider $r =\lc r_{1}\rc^{(n_{1})}\cdots \lc r_{m_{r}}\rc^{(n_{m_{r}})}\in\frak R''$ with $n=n_{1}+\cdots+n_{m_{r}}=k+1$. Then
\begin{align*}
\bar{f}(P(r))
&=\bar{f}(P(\lc r_{1}\rc^{(n_{1})} \cdots \lc r_{m_{r}}\rc^{(n_{m_{r}})}))\\
&=\bar{f}\left(\frac{1}{m_{r}-1}\left(\sum^{m_r}_{i=1}P(r^{\ast}_{i})-r\right)\right)
\quad(\text{by Eq. }(\mref{eq:R''}))\\
&=\frac{1}{m_{r}-1}\left(\sum^{m_r}_{i=1}\bar{f}(P(r^{\ast}_{i}))-\bar{f}(r)\right)
\quad(\text{$\bar{f}$ is an algebra homomorphism})\\
&=\frac{1}{m_{r}-1}\left(\sum^{m_r}_{i=1}Q(\bar{f}(r^{\ast}_{i}))-\bar{f}(r)\right)
\quad(\text{by the induction hypothesis })\\
&=\frac{1}{m_{r}-1}\left(\sum^{m_r}_{i=1}Q(Q^{n_{1}}(\bar{f}(r_{1}))\cdots Q^{n_{i-1}}(\bar{f}(r_{i-1})) Q^{n_{i}-1}(\bar{f}(r_{i}))Q^{n_{i+1}}(\bar{f}(r_{i+1}))\cdots Q^{n_{m_r}}(\bar{f}(r_{m_r}))\right)\\
&~~~~~~~~-\frac{1}{m_{r}-1}Q^{n_{1}}(\bar{f}(r_{1}))\cdots Q^{n_{m_r}}(\bar{f}(r_{m_r}))
\quad(\text{$\bar{f}$ is an algebra homomorphism})\\
&=Q(Q^{n_{1}}(\bar{f}(r_{1}))\cdots Q^{n_{m_r}}(\bar{f}(r_{m_r})))
\quad(\text{$Q$ is a Reynolds operator})\\
&=Q(\bar{f}(\lc r_{1}\rc^{(n_{1})})\cdots \bar{f}(\lc r_{m_{r}}\rc^{(n_{m_{r}})}))
\quad(\text{$Q$ is a Reynolds operator})\\
&=Q(\bar{f}(\lc r_{1}\rc^{(n_{1})}\cdots\lc r_{m_{r}}\rc^{(n_{m_{r}})}))
\quad(\text{$\bar{f}$ is an algebra homomorphism})\\
&=Q(\bar{f}(r)).
\end{align*}
Therefore, $\bar{f}$ is a Reynolds algebra homomorphism.

Finally, we verify the uniqueness of $\bar{f}=\dirlim\bar{f}_{n}$ in the universal property.
Assume that there is another Reynolds algebra homomorphism $\bar{f'}: \bfk\mathfrak R\rightarrow A$ such that
$$\bar{f}'i=f =\bar{f}i.$$
Define two maps by:
$$ \bar{f}'_{0}:=\bar{f}'|_{\bfk\frak R_{0}}: \bfk\frak R_{0}\rightarrow A,\quad  \bar{f}'_{1}:=\bar{f}'|_{\bfk\frak R_{1}}: \bfk\frak R_{1}\rightarrow A.
$$
Then for all $x\in X$, we have
$$
\bar{f}'_{0}i_{0}(x)=f(x)=\bar{f}_{0}i_{0}(x),\quad
\bar{f}_{1}'P(x)=Q(f'(x))=Q(\bar{f}_{0}'(x)).
$$
The universal property of $\bfk\frak R$ gives $\bar{f}'_{0}=\bar{f}_{0}.$
Furthermore, define
$$\{\bar{f}'_{n}:=\bar{f}'|_{\bfk\frak R_{n}}:\bfk\frak R_{n}\to A\mid n\geq0\}.$$
From its construction and the fact that $\bar{f}'$ is a Reynolds algebra homomorphism, we obtain
\begin{align*}
\bar{f}'_{n+1}i_{n+1}(r_{n})&=f(r_{n})=\bar{f}_{n+1}i_{n+1}(r_{n}),\\
\bar{f}_{n+1}'P(r_{n})&=Q(f'(r_{n}))=Q(\bar{f}_{n}'(r_{n})) \tforall r_{n}\in\bfk\frak R_{n}.
\end{align*}
By our inductive construction of $\{\bar{f}_n\mid n\geq 0\}$, such $\bar{f}_n$ are unique. Thus for all $n\geq0$
$$\bar{f}'_{n}=\bar{f}_{n}.$$
Therefore $\bar{f}'=\bar{f}$. This proves the uniqueness of $\bar{f}$. The proof of Theorem~\mref{thm:main} is completed.
\end{proof}

\subsection{The proof of Lemma~\mref{lem:C}}
\mlabel{ss:lemma}

We now give the proof of Lemma~\mref{lem:C}, which is divided into the following three cases of the inclusions of $r$ and $s$ in $\frakR''$.
\begin{enumerate}
\item
$r$ is in $\frakR''$ and $s$ is in $\frakR'$;
\mlabel{it:cr}
\item
$s$ is in $\frakR''$ and $r$ is in $\frakR'$;
\mlabel{it:cs}
\item
both $r$ and $s$ are in $\frak R''$.
\mlabel{it:crs}
\end{enumerate}

\noindent
{\bf Case (\mref{it:cr}). $r\in \frakR''$ and $s\in\frak R'$:} Then $r=\lc r_{1}\rc^{(n_{1})} \cdots \lc r_{m_{r}}\rc^{(n_{m_{r}})}\in\frak R''$ with $r_{i}\in\frak R'\setminus\lc\frak R'\rc, n_{i}\geq1$, where $1\leq i\leq m_{r}, m_r\geq 2$.
Since $s\in\frak R'$, we have $P(s)=\lc s\rc$ by Eq.~(\mref{eq:R'}). Therefore
\begin{align*}
m_{r}P\bigg(r P(s)\bigg)+r P(s)-P(r s)
&=m_{r}P(\lc r_{1}\rc^{(n_{1})} \cdots \lc r_{m_{r}}\rc^{(n_{m_{r}})}\rc\lc s\rc)+r \lc s\rc-P(r s)\\
&=\left(\sum_{i=1}^{m_{r}}P(r^{\ast}_{i}\lc s\rc)
+P(r s)-r \lc s\rc\right)+r \lc s\rc-P(r s)
\quad(\text{by Eq. }(\mref{eq:R''}))\\
&=\sum_{i=1}^{m_{r}}P\bigg(r^{\ast}_{i}P(s)\bigg).
\end{align*}

\noindent
{\bf Case (\mref{it:cs}). $r\in\frak R'$ and $s\in \frakR''$: }
The proof is similar to that of Case (\mref{it:cr}).

\noindent
{\bf Case (\mref{it:crs}). $r$ and $s$ are both in $\frakR''$:}
Then $r =\lc r_{1}\rc^{(n_{1})} \cdots \lc r_{m_{r}}\rc^{(n_{m_{r}})}$ and $s=\lc s_{1}\rc^{(l_{1})} \cdots \lc s_{m_{s}}\rc^{(l_{m_{s}})}$ with $r_{i}, s_{j}\in\frak R'\setminus\lc\frak R'\rc, n_{i}\geq1,l_{j}\geq1$, where $1\leq i\leq m_{r},1\leq j\leq m_{s}$ and $m_r, m_s\geq 2$.

We first verify Eq.~(\mref{eq:SE}) by induction on $q:=\sum\limits^{m_{r}}_{i=1}n_{i}$. Then $q\geq m_{r}\geq 2$. When $q=2$, we obtain $m_{r}=2$ and $n_{1}=n_{2}=1$. Thus
\begin{align*}
m_{s}P\bigg(P(r)s\bigg)
&=m_{s}P\bigg(P(\lc r_{1}\rc\lc r_{2}\rc)s\bigg)\\
&=m_{s}P(P(r_{1}\lc r_{2}\rc)s)+m_{s}P(P(\lc r_{1} \rc r_{2})s)-m_{s}P(\lc r_{1} \rc\lc r_{2}\rc s)
\quad(\text{by Eq. }(\mref{eq:R''}))\\
&=m_{s}P(\lc r_{1}\lc r_{2}\rc\rc s)+m_{s}P(\lc\lc r_{1} \rc r_{2}\rc s)-m_{s}P(\lc r_{1} \rc\lc r_{2}\rc s)\quad(\text{by Eq. }(\mref{eq:R'}))\\
&=P(r_{1}\lc r_{2}\rc s)+\sum_{j=1}^{m_{s}}P(\lc r_{1}\lc r_{2}\rc\rc s^{\ast}_{j})-\lc r_{1}\lc r_{2}\rc\rc s
+P(\lc r_{1} \rc r_{2}s)+\sum_{j=1}^{m_{s}}P(\lc\lc r_{1} \rc r_{2}\rc s^{\ast}_{j})-\lc\lc r_{1} \rc r_{2}\rc s\\
&~~~~~~~~-\frac{m_{s}}{m_{s}+1}P(r_{1}\lc r_{2}\rc s)
-\frac{m_{s}}{m_{s}+1}P(\lc r_{1}\rc r_{2}s)
-\frac{m_{s}}{m_{s}+1}\sum^{m_{s}}_{j=1}P(\lc r_{1}\rc\lc r_{2}\rc s^{\ast}_{j})
+\frac{m_{s}}{m_{s}+1}\lc r_{1}\rc\lc r_{2}\rc s\\
&\hspace{2cm}(\text{by Eq. }(\mref{eq:R''}))\\
&=\lc r_{1}\lc r_{2}\rc s\rc+\sum_{j=1}^{m_{s}}P(\lc r_{1}\lc r_{2}\rc\rc s^{\ast}_{j})-\lc r_{1}\lc r_{2}\rc\rc s
+\lc\lc r_{1} \rc r_{2}s\rc+\sum_{j=1}^{m_{s}}P(\lc\lc r_{1} \rc r_{2}\rc s^{\ast}_{j})-\lc\lc r_{1} \rc r_{2}\rc s\\
&~~~~~~~~-\frac{m_{s}}{m_{s}+1}\lc r_{1}\lc r_{2}\rc s\rc
-\frac{m_{s}}{m_{s}+1}\lc\lc r_{1}\rc r_{2}s\rc
-\frac{m_{s}}{m_{s}+1}\sum^{m_{s}}_{j=1}P(\lc r_{1}\rc\lc r_{2}\rc s^{\ast}_{j})
+\frac{m_{s}}{m_{s}+1}\lc r_{1}\rc\lc r_{2}\rc s\\
&\hspace{2cm}(\text{by Eq. }(\mref{eq:R'}))\\
&=\sum_{j=1}^{m_{s}}P(\lc r_{1}\lc r_{2}\rc\rc s^{\ast}_{j})
+\sum_{j=1}^{m_{s}}P(\lc\lc r_{1} \rc r_{2}\rc s^{\ast}_{j})
-\frac{m_{s}}{m_{s}+1}\sum^{m_{s}}_{j=1}P(\lc r_{1}\rc\lc r_{2}\rc s^{\ast}_{j})
+\frac{m_{s}}{m_{s}+1}\lc r_{1}\rc\lc r_{2}\rc s\\
&~~~~~~~~+\frac{1}{m_{s}+1}\lc r_{1}\lc r_{2}\rc s\rc
+\frac{1}{m_{s}+1}\lc\lc r_{1}\rc r_{2}s\rc
-\lc r_{1}\lc r_{2}\rc\rc s-\lc\lc r_{1} \rc r_{2}\rc s
\end{align*}
and
\begin{align*}
&\sum_{j=1}^{m_{s}}P\bigg(P(r)s^{\ast}_{j}\bigg)+P(r s)-P(r)s\\
&=\sum_{j=1}^{m_{s}}P(\lc r_{1}\lc r_{2}\rc\rc s^{\ast}_{j})+\sum_{j=1}^{m_{s}}P(\lc\lc r_{1}\rc r_{2}\rc s^{\ast}_{j})-\sum_{j=1}^{m_{s}}P(\lc r_{1}\rc\lc r_{2}\rc s^{\ast}_{j})
+\frac{1}{m_{s}+1}P(r_{1}\lc r_{2}\rc s)
+\frac{1}{m_{s}+1}P(\lc r_{1}\rc r_{2}s)\\
&~~~~~~~~+\frac{1}{m_{s}+1}\sum^{m_{s}}_{j=1}P(\lc r_{1}\rc\lc r_{2}\rc s^{\ast}_{j})
-\frac{1}{m_{s}+1}\lc r_{1}\rc\lc r_{2}\rc s
-\lc\lc r_{1}\rc r_{2}\rc s
-\lc r_{1}\lc r_{2}\rc\rc s
+\lc r_{1}\rc \lc r_{2}\rc s
\ \ (\text{by Eq. }(\mref{eq:R''}))\\
&=\sum_{j=1}^{m_{s}}P(\lc r_{1}\lc r_{2}\rc\rc s^{\ast}_{j})+\sum_{j=1}^{m_{s}}P(\lc\lc r_{1}\rc r_{2}\rc s^{\ast}_{j})-\sum_{j=1}^{m_{s}}P(\lc r_{1}\rc\lc r_{2}\rc s^{\ast}_{j})
+\frac{1}{m_{s}+1}\lc r_{1}\lc r_{2}\rc s\rc
+\frac{1}{m_{s}+1}\lc\lc r_{1}\rc r_{2}s\rc\\
&~~~~~~~~+\frac{1}{m_{s}+1}\sum^{m_{s}}_{j=1}P(\lc r_{1}\rc\lc r_{2}\rc s^{\ast}_{j})
-\frac{1}{m_{s}+1}\lc r_{1}\rc\lc r_{2}\rc s-\lc\lc r_{1}\rc r_{2}\rc s
-\lc r_{1}\lc r_{2}\rc\rc s
+\lc r_{1}\rc \lc r_{2}\rc s
\ \ (\text{by Eq. }(\mref{eq:R'}))\\
&=\sum_{j=1}^{m_{s}}P(\lc r_{1}\lc r_{2}\rc\rc s^{\ast}_{j})
+\sum_{j=1}^{m_{s}}P(\lc\lc r_{1}\rc r_{2}\rc s^{\ast}_{j})
-\frac{m_{s}}{m_{s}+1}\sum^{m_{s}}_{j=1}P(\lc r_{1}\rc\lc r_{2}\rc s^{\ast}_{j})
+\frac{m_{s}}{m_{s}+1}\lc r_{1}\rc\lc r_{2}\rc s\\
&~~~~~~~~+\frac{1}{m_{s}+1}\lc r_{1}\lc r_{2}\rc s\rc
+\frac{1}{m_{s}+1}\lc\lc r_{1}\rc r_{2}s\rc
-\lc\lc r_{1}\rc r_{2}\rc s
-\lc r_{1}\lc r_{2}\rc\rc s.
\end{align*}
Hence Eq.~(\mref{eq:SE}) holds. Assume that Eq.~(\mref{eq:SE}) holds for $q\leq p$ with $p\geq m_{r}$ and consider $q=p+1$. By Eq.~(\mref{eq:R''}), we have
\begin{align}
P(r)=\frac{1}{m_{r}-1}\left(\sum^{m_{r}}_{i=1}P(r^{\ast}_{i})-r \right).
\mlabel{eq:add}
\end{align}
Thus
\begin{align*}
&m_{s}P\bigg(P(r)s\bigg)\\
&=\frac{m_{s}}{m_{r}-1}\sum^{m_{r}}_{i=1}P\bigg(P(r^{\ast}_{i})s\bigg) -\frac{m_{s}}{m_{r}-1}P(rs)\quad(\text{by Eq.~(\mref{eq:add})})\\
&=\frac{1}{m_{r}-1}\sum^{m_{r}}_{i=1}\sum_{j=1}^{m_{s}}P\bigg(P(r^{\ast}_{i})s^{\ast}_{j}\bigg)
+\frac{1}{m_{r}-1}\sum^{m_{r}}_{i=1}P(r^{\ast}_{i}s)
-\frac{1}{m_{r}-1}\sum^{m_{r}}_{i=1}P(r^{\ast}_{i})s
-\frac{m_{s}}{m_{r}-1}P(rs)\\
&\hspace{2cm}(\text{by the induction hypothesis})\\
&=\frac{1}{m_{r}-1}\sum^{m_{r}}_{i=1}\sum_{j=1}^{m_{s}}P\bigg(P(r^{\ast}_{i})s^{\ast}_{j}\bigg)
+\frac{1}{m_{r}-1}\sum^{m_{r}}_{i=1}P(r^{\ast}_{i}s)
-\frac{1}{m_{r}-1}\sum^{m_{r}}_{i=1}P(r^{\ast}_{i})s\\
&~~~~~~~~-\frac{m_{s}}{m_{r}-1}\frac{1}{m_{r}+m_{s}-1}\sum^{m_{r}}_{i=1}P(r^{\ast}_{i}s)
-\frac{m_{s}}{m_{r}-1}\frac{1}{m_{r}+m_{s}-1}\sum^{m_{s}}_{j=1}P(rs^{\ast}_{j})
+\frac{m_{s}}{m_{r}-1}\frac{1}{m_{r}+m_{s}-1}rs\\
&\hspace{2cm}(\text{by Eq. }(\mref{eq:R''}))\\
&=\frac{1}{m_{r}-1}\sum^{m_{r}}_{i=1}\sum_{j=1}^{m_{s}}P\bigg(P(r^{\ast}_{i})s^{\ast}_{j}\bigg)
+\frac{1}{m_{r}+m_{s}-1}\sum^{m_{r}}_{i=1}P(r^{\ast}_{i}s)
-\frac{1}{m_{r}-1}\sum^{m_{r}}_{i=1}P(r^{\ast}_{i})s\\
&~~~~~~~~-\frac{m_{s}}{m_{r}-1}\frac{1}{m_{r}+m_{s}-1}\sum^{m_{s}}_{j=1}P(rs^{\ast}_{j})
+\frac{m_{s}}{m_{r}-1}\frac{1}{m_{r}+m_{s}-1}rs
\end{align*}
and
\begin{align*}
&\sum_{j=1}^{m_{s}}P\bigg(P(r)s^{\ast}_{j}\bigg)+P(r s)-P(r)s\\
&=\frac{1}{m_{r}-1}\sum^{m_{r}}_{i=1}\sum_{j=1}^{m_{s}}
P\bigg(P(r^{\ast}_{i})s^{\ast}_{j}\bigg)
-\frac{1}{m_{r}-1}\sum_{j=1}^{m_{s}}P(rs^{\ast}_{j})
+P(r s)
-\frac{1}{m_{r}-1}\sum^{m_{r}}_{i=1}P(r^{\ast}_{i})s
+\frac{1}{m_{r}-1}rs\\
&\hspace{2cm}(\text{by Eq.~(\mref{eq:add})})\\
&=\frac{1}{m_{r}-1}\sum^{m_{r}}_{i=1}\sum_{j=1}^{m_{s}}
P\bigg(P(r^{\ast}_{i})s^{\ast}_{j}\bigg)
-\frac{1}{m_{r}-1}\sum_{j=1}^{m_{s}}P(rs^{\ast}_{j})
-\frac{1}{m_{r}-1}\sum^{m_{r}}_{i=1}P(r^{\ast}_{i})s
+\frac{1}{m_{r}-1}rs\\
&~~~~~~~~+\frac{1}{m_{r}+m_{s}-1}\sum^{m_{r}}_{i=1}P(r^{\ast}_{i}s)
+\frac{1}{m_{r}+m_{s}-1}\sum^{m_{s}}_{j=1}P(rs^{\ast}_{j})
-\frac{1}{m_{r}+m_{s}-1}rs\quad(\text{by Eq. }(\mref{eq:R''}))\\
&=\frac{1}{m_{r}-1}\sum^{m_{r}}_{i=1}\sum_{j=1}^{m_{s}}
P\bigg(P(r^{\ast}_{i})s^{\ast}_{j}\bigg)
-\frac{m_{s}}{m_{r}-1}\frac{1}{m_{r}+m_{s}-1}\sum_{j=1}^{m_{s}}P(rs^{\ast}_{j})
-\frac{1}{m_{r}-1}\sum^{m_{r}}_{i=1}P(r^{\ast}_{i})s\\
&~~~~~~~~+\frac{1}{m_{r}+m_{s}-1}\sum^{m_{r}}_{i=1}P(r^{\ast}_{i}s)
+\frac{m_{s}}{m_{r}-1}\frac{1}{m_{r}+m_{s}-1}rs.
\end{align*}
Therefore
$$
m_{s}P\bigg(P(r)s\bigg)=\sum_{j=1}^{m_{s}}P\bigg(P(r)s^{\ast}_{j}\bigg)+P(r s)-P(r)s.
$$
This completes the inductive proof of Eq.~\eqref{eq:SE}. A similar argument proves Eq.~(\mref{eq:SE'}).

Now the proof of Lemma~\mref{lem:C} is completed.

\smallskip

\noindent {\bf Acknowledgements}: This work is supported by the National Natural Science Foundation of China (Grant Nos. 11771190 and 11861051) and the Natural Science Foundation of Ningxia (Grant No. 2018AAC03063). The authors thank William Sit for helpful discussions.

\end{document}